\documentclass[10pt]{article}

\usepackage{amsmath}
\usepackage{amsfonts}
\usepackage{amssymb}
\usepackage{framed}
\usepackage{mathtools}
\usepackage{enumerate}
\usepackage{mathrsfs}
\usepackage[hidelinks]{hyperref}
\usepackage{enumitem}
\usepackage{mathrsfs}
\usepackage{scrextend}
\usepackage{amsthm}
\usepackage{bbold}
\usepackage{bbm}
\usepackage{comment}

\usepackage[T1]{fontenc}
\usepackage{lmodern}

\usepackage{thmtools}
\usepackage{thm-restate}
\usepackage{mathabx}
\usepackage[margin=1.5in]{geometry}
\usepackage{rotating}

\usepackage[nameinlink]{cleveref}

\usepackage[utf8]{inputenc}
\usepackage[T1]{fontenc}

\usepackage{xcolor}

\usepackage{tikz-cd}

\numberwithin{equation}{section}


\theoremstyle{plain}
\newtheorem{thm}{Theorem}[section]
\newtheorem*{thm*}{Theorem}

\newtheorem{lem}[thm]{Lemma}

\newtheorem{cor}[thm]{Corollary}
\newtheorem{rem}[thm]{Remark}

\theoremstyle{definition}
\newtheorem{defn}[thm]{Definition}
\newtheorem{conj}[thm]{Conjecture}

\theoremstyle{remark}

\usepackage[nottoc,numbib]{tocbibind}

\tikzset{
  symbol/.style={
    draw=none,
    every to/.append style={
      edge node={node [sloped, allow upside down, auto=false]{$#1$}}}
  },
    labl/.style={anchor=south, rotate=90, inner sep=.5mm}
}

\newcommand\restr[2]{{
	\left.\kern-\nulldelimiterspace
	#1
	\vphantom{\big|}
	\right|_{#2}
	}}

\newcommand{\an}{\operatorname{an}}

\newcommand{\Aut}{\textrm{Aut}}

\newcommand{\ch}[1]{\widecheck{{#1}}}

\newcommand{\HL}{\operatorname{HL}}

\newcommand{\MT}{\mathbf{MT}}

\usepackage{tikz-cd}
\usetikzlibrary{cd}
\usepackage{leftidx}
\usetikzlibrary{positioning}
\usepackage{dsfont}
\usepackage{tikz}
\usetikzlibrary{matrix,arrows,decorations.pathmorphing,positioning}

\newcommand{\GL}{\operatorname{GL}}

\usepackage{amsmath,calligra,mathrsfs}

\usepackage[cal=boondoxo]{mathalfa}

 \newcommand{\Addresses}{{
  \bigskip
  \footnotesize

\textsc{CNRS, IMJ-PRG, Sorbonne Universit\'{e}, 4 place Jussieu, 75005 Paris, France}\par\nopagebreak
  \textit{E-mail address}, G.~Baldi: \texttt{baldi@imj-prg.fr} 

  \medskip

 \textsc{The Weizmann Institute of Science, Rehovot, Israel}\par\nopagebreak
  \textit{E-mail address}, G.~Binyamini: \texttt{gal.binyamini@weizmann.ac.il} 

  \medskip 
  
\textsc{Rice University. 6100 Main St MS 136, Houston, TX 77005 (United States)}\par\nopagebreak
  \textit{E-mail address}, D.~Urbanik: \texttt{du8@rice.edu}
}}

\DeclareMathOperator{\sheafhom}{\mathcal{H \kern -1pt o \kern -2pt m}}
\DeclareMathOperator{\sheafper}{\mathcal{P \kern -1pt e \kern -2pt r}}
\DeclareMathOperator{\sheafiso}{\mathcal{I \kern -1pt s \kern -2pt o}}
\DeclareMathOperator{\sheafend}{\mathcal{E \kern -1pt n \kern -2pt d}}
\DeclareMathOperator{\sheafaut}{\mathcal{A \kern -1pt u \kern -2pt t}}

\newcommand{\C}{\mathbb{C}}
\tikzset{
  trim node/.default=1cm,
  trim node/.style={
    overlay,
    append after command={
      ([xshift={+#1}]\tikzlastnode.north west)
      ([xshift={+-#1}]\tikzlastnode.south east)}},
  down and trim/.default=1cm,
  down and trim/.style={
    yshift=-(\pgfmatrixcurrentcolumn-1)*1.5\baselineskip,
    trim node={#1}},
  downup and trim/.default=1cm,
  downup and trim/.style={
    yshift=iseven(\pgfmatrixcurrentcolumn) ? -1.5\baselineskip : 0pt,
    trim node={#1}},
  -|/.style={to path={-|(\tikztotarget)\tikztonodes}},
  |-/.style={to path={|-(\tikztotarget)\tikztonodes}},
  -| sl/.style={-|, xslant=-1},
  |- sl/.style={|-, xslant= 1},
  center picture/.style={
    trim left=(current bounding box.center),
    trim right=(current bounding box.center)}}

\newcommand{\Qbar}{\overline{\mathbb{Q}}}

\newcommand{\Q}{\mathbb{Q}}
\newcommand{\N}{\mathbb{N}}
\newcommand{\Z}{\mathbb{Z}}

\newcommand{\V}{\mathbb{V}}

\newcommand{\PP}{\mathbb{P}}

\usepackage{microtype}

\title{Algebraic Hodge generic points are dense}\date{\today}
\author{Gregorio Baldi, Gal Binyamini, and David Urbanik}

\begin{document}

\maketitle

\begin{abstract}
Let $f: X \to S$ be a quasi-projective family of varieties defined over $\overline{\mathbb{Q}} \subset \mathbb{C}$. We show that the points of $S(\overline{\mathbb{Q}})$ that are Hodge generic for the variation of Hodge structures associated to $f$ are analytically dense in $S(\mathbb{C})$. In fact, in the spirit of the Grothendieck period conjecture and under a large monodromy assumption, we prove the density of the points of $S(\overline{\mathbb{Q}})$ where the periods of the fibre do not  satisfy extra relations ``up to degree $\delta$''. As a by-product, we also establish new instances of the Mumford-Tate conjecture, beyond the realm of abelian motives. When the base $S$ is a curve, we provide quantitative estimates for points satisfying these properties.

The main technical contribution is a new result on relations satisfied by solutions of $G$-operators, which relies on height estimates due to Bombieri and Andr\'e.
\end{abstract}

\tableofcontents

\section{Introduction}\label{sec1}
The main result of this paper, \Cref{mainhodgegenthm}, is the abundance of $\Qbar$-Hodge generic points for every graded-polarizable variation of mixed Hodge structures associated to a family of quasi-projective varieties defined over $\Qbar$. In fact, this result is deduced from a general theorem on $G$-operators: \autoref{thmwithproof}. Before discussing these results further, we record some applications to the well-known \emph{motivic conjectures}. To keep the presentation as accessible as possible, rather than presenting the most general setting, we focus on the case of smooth hypersurfaces in projective space.

\subsection{Main motivic results}\label{mmr}

Let $Y$ be a smooth projective variety defined over a number field $K \subset \overline{\mathbb{Q}} \subset \mathbb{C}$. Fix integers $j,\ell,\delta, m \in \mathbb{N}$, with $\ell$ prime. We consider the following properties which will be described in detail in the next section; see also \cite[Sec. 7]{zbMATH02128361} for an introduction to such conjectures. 

\begin{align*}
\mathrm{P}(Y,j,\delta) \; &:\;\; \text{The period conjecture in cohomological degree $j$ holds for } Y \text{ up to degree } \delta; \\[1em]
\mathrm{MT}(Y,\ell,j) \; &:\;\; \text{The $\ell$-adic Mumford--Tate conjecture holds for } Y \text{ in cohomological degree } j; \\[1em]
\mathrm{H}(Y,j,m) \; &:\;\; \text{For all $c\geq 0$ and all $a,b$ such that $a+b \leq m$, every Hodge class in } \\
&\quad \left(H^j(Y_{\mathbb{C}},\mathbb{Q})^{\otimes a} \otimes H^j(Y_{\mathbb{C}},\mathbb{Q})^{\vee\,\otimes b} \otimes \mathbb{Q}(c) \right) 
\text{ is in the $\Q$-span of algebraic classes;} \\[0.5em]
\mathrm{T}(Y,\ell,j,m) \; &:\;\; \text{For all $c\geq 0$ and all $a,b$ such that $a+b \leq m$, every Tate class in } \\
&\quad \left(H^j(Y_{\Qbar},\mathbb{Q}_\ell)^{\otimes a} \otimes H^j(Y_{\Qbar},\mathbb{Q}_\ell)^{\vee\,\otimes b} \otimes \mathbb{Q}_{\ell}(c) \right) 
\text{ is in the $\Q_\ell$-span of algebraic classes.} \\
\end{align*}
\noindent We also write $H(Y, j)$ (resp. $T(Y, \ell, j)$) to refer to the condition that $H(Y, j, m)$ holds for all $m$ (resp. $T(Y, \ell, j, m)$ holds for all $m$).

 The precise meaning of \emph{the period conjecture up to degree $\delta$} is given in \S \ref{periodconj}, see indeed \Cref{degdeltaperconj}, \Cref{galperiodconj}, and \Cref{mainthmonperiod}. Moreover, if $\ell$ is not prime, we define $\mathrm{T}(Y,\ell,j)$ and $\mathrm{MT}(Y,\ell,j)$ to be the conjunction of the corresponding statements for all prime divisors of $\ell$. Observe that, under $\mathrm{MT}(Y,\ell,j)  $,  $\mathrm{T}(Y,\ell,j)$ and $\mathrm{H}(Y,j)$ are equivalent. 

Let $f: X \to S$ be a smooth projective family defined over a number field $K\subset \overline{\mathbb{Q}} \subset \mathbb{C}$, and let $j,m, \delta, \ell \in \mathbb{N}$. For $s \in S(\overline{\mathbb{Q}})$, write $X_s$ for the fibre. We define the corresponding loci in $S$ by
\begin{align*}
S_{P,j,\delta} \; &:=\; \{\, s \in S(\overline{\mathbb{Q}}) \mid \mathrm{P}(X_s,j,\delta) \,\},\\
S_{\mathrm{H},j,m} \; &:=\; \{\, s \in S(\overline{\mathbb{Q}}) \mid \mathrm{H}(X_s,j,m) \,\}, \\
S_{\mathrm{H},j} \; &:=\; \{\, s \in S(\overline{\mathbb{Q}}) \mid \mathrm{H}(X_s,j) \  \,\}, \\
S_{\mathrm{T},\ell,j,m} \; &:=\; \{\, s \in S(\overline{\mathbb{Q}}) \mid \mathrm{T}(X_s,\ell,j,m) \,\}, \\
S_{\mathrm{T},\ell,j} \; &:=\; \{\, s \in S(\overline{\mathbb{Q}}) \mid \mathrm{T}(X_s,\ell,j)  \,\}, \\
S_{\mathrm{MT},\ell,j} \; &:=\; \{\, s \in S(\overline{\mathbb{Q}}) \mid \mathrm{MT}(X_s,\ell,j) \,\}, 
\end{align*}
and the loci where all the properties hold simultaneously by (for fixed $j,\ell, \delta$):
\begin{equation}
\label{Smotdef}
S_{\operatorname{Motivic}, j,\ell, \delta,m} \;:=\;  S_{P,j,\delta}  \;\cap\;S_{\mathrm{H},j,m} \;\cap\; S_{\mathrm{T},\ell,j,m} \;\cap\; S_{\mathrm{MT},\ell,j}
\end{equation}
\begin{equation}
\label{Smotdef2}
S_{\operatorname{Motivic}, j,\ell, \delta} \;:=\;  S_{P,j,\delta}  \;\cap\;S_{\mathrm{H},j} \;\cap\; S_{\mathrm{T},\ell,j} \;\cap\; S_{\mathrm{MT},\ell,j}
\end{equation}

This is the locus where the (tannakian versions, possibly \emph{up to $m$}, of the) Hodge and  Tate conjectures hold, as well as the $\ell$-adic Mumford-Tate conjecture in cohomological degree $j$, and the Grothendieck period conjecture up to degree $\delta$ and cohomological degree $j$. We refer to it as the \emph{(arithmetic) motivic locus of $f$} (in degree $j$). Conjecturally $S_{\operatorname{Motivic}, j,\ell, d}= S(\Qbar) $, and we are interested in weaker results that say that such a locus is \emph{big} in $S$. The first representative result is the following:

\begin{thm}\label{mainthmhyper}
Let $n$ and $d$ be natural numbers with $n \geq 1, d\geq 3$, and $S=U_{n, d}$ be the ($\mathbb{Q}$-algebraic) parameter space of smooth hypersurfaces of
$\mathbb{P}^{n+1}$ of degree $d$. If $n$ is even, then, for any $\ell, \delta $ and $m < \dim H^n(X,\Q)_{\operatorname{prim}} $ the locus 
$$S_{\operatorname{Motivic}, n,\ell, \delta,m}\subset S(\Qbar)$$
is analytically dense in $S(\C)$. If $n$ is odd, then, for any $\ell$ and $\delta$ the locus 
$$S_{\operatorname{Motivic}, n,\ell, \delta}\subset S(\Qbar)$$
is analytically dense in $S(\C)$.
\end{thm}

A more precise version can be given for any \emph{Lefschetz pencil} of hypersurfaces, cf. \cite[Exposé XVII]{zbMATH03407568} for an introduction. To avoid distinguishing between even and odd dimensions, as we did in the previous statement, we just give an application to the Period and Mumford-Tate conjectures. (In fact, the same conclusion applies to any pencil of hypersurfaces in $\mathbb{P}^{n+1}$ with \emph{large monodromy}.)

\begin{thm}\label{thmLefschetzpen}
Let $f : X \to \mathbb{P}^1$ be a Lefschetz pencil (smooth outside the singular locus $\Sigma \subset \mathbb{P}^1$) of hypersurfaces of dimension $n$ and degree $d$, defined over a number field $K$. Then for any positive integers $\delta, \ell > 0$ the set
\[ \{ t \in \mathbb{N} : [t : 1] \in (\mathbb{P}^1 \setminus \Sigma)_{\mathrm{P},n,\delta} \cap(\mathbb{P}^1 \setminus \Sigma)_{\mathrm{MT},\ell,n} \} \]
has natural density one in $\mathbb{N}$.
\end{thm}

\subsection{Algebraic Hodge generic points}
\label{Hodgegenptssec}
Let $S$ be a smooth complex algebraic variety. Given a graded polarizable (rational) variation of mixed Hodge structures (referred to as a $\Q$VHS from now on) $\V$ on $S$ there is a notion of \emph{Mumford-Tate group of $\V$} as well as a \emph{Mumford-Tate group at each $s\in S(\C)$}, which we denote, respectively, by
\begin{displaymath}
\mathbf{G}=\MT(\V) \\ \ \ \text{  and } \ \ \ \     \mathbf{G}_s=\MT(\V_s);
\end{displaymath}
see \Cref{MTdef} and \cite{2021arXiv210708838B} for more details. A point $s\in S(\C)$ is said to be \emph{Hodge generic} for $(S,\V)$ if its Mumford-Tate group is ``as big as possible'', i.e., if $ \mathbf{G}_s= \mathbf{G}$. In \Cref{sec5} we give an equivalent description of this property using period maps and period domains.

All $\mathbb{Q}$VHSs appearing in our paper will be derived from the natural $\mathbb{Q}$VHSs $R^jf_*\Q$ through some combination of passing to primitive cohomology, taking direct sums, tensor powers, and extensions. (For an introduction to the \emph{primitive cohomology} we refer to Voisin's book \cite{zbMATH05221670}. Given a VHS $\V$ arising in the cohomology of a smooth projective family, its primitive part will be denoted $\V_{\operatorname{prim}}$.) We refer to \S \ref{setting} for more details (especially about the mixed case) and the notion of $K$-geometric VHS.

Since Hodge theory is at its core a transcendental theory, understanding its properties over number fields can be challenging. The main result of the paper is the following, which was previously known only in the pure case and under the absolute Hodge conjecture, cf. \cite[Thm. 1.11]{zbMATH06492665}. (See also \Cref{relworksec} for a more detailed discussion of related work.)
\begin{thm}
\label{mainhodgegenthm}
Let $f : X \to S$ be a quasi-projective family defined over a number field $K$. Fix an integer $j \geq 0$, and set $\V = R^{j} f_{*} \mathbb{Q}$ to be the associated graded polarizable $\mathbb{Q}$VHS (see \S\ref{VHSconstrsec}). Then the locus of Hodge generic points in $S(\Qbar)$ for $\V$ is analytically dense in $S(\mathbb{C})$. Put differently, the locus of points in $S(\Qbar)$ that do not lie in the (tensorial) Hodge locus of $\V$ is analytically dense in $S(\mathbb{C})$.
\end{thm}

\begin{rem}
As a corollary of \Cref{mainhodgegenthm} one immediately gets the following. For infinitely many $a,b\in \N$, the group $\mathbf{G}=\operatorname{GSO}(2a,b)$ is the Mumford-Tate group of the degree two primitive cohomology of a projective surface defined over $\Qbar$. In particular, this holds for all $a$ (resp. $b$) that can be realized as $h^{2,0}$ (resp. $h^{1,1}-1$) of a smooth complete intersection surface. For $a>1$ and $b>2$, $\mathbf{G}$ is not of Hermitian type, hence not associated to any Shimura variety, so we believe that this result was not known before for such $a$ and $b$.  
\end{rem}

\noindent In fact, in \Cref{generalHggenthm2}, we will show the existence of some integer $d$ such that \autoref{mainhodgegenthm} holds even if we limit ourselves to just those points $s \in S(\Qbar)$ whose field of definition\footnote{Throughout the paper we write $\kappa(s)$ to denote the minimal extension of $K$ inside $\Qbar$ over which a given point $s \in S(\Qbar)$ is defined.} $\kappa(s)$ satisfies the bound $[\kappa(s) : K] \leq d$. We also give a \emph{density one} result in the spirit of \autoref{thmLefschetzpen} when $S$ is a Zariski open of $\PP^1$, cf. \Cref{quantitativeHodgegen}. 

A related question is whether one can produce points in $S(\Qbar)$ that are certifiably Hodge generic -- that is, whether one can make the density of \autoref{mainhodgegenthm} effective. This seems to be beyond the reach of our method. However, we can prove the following result, which effectively produces a subset of points, ``99 percent'' of which are Hodge generic. For a more precise statement, see \autoref{effcor}. 
\begin{thm}\label{effectivethmintro}
In the setting of \autoref{mainhodgegenthm} there is, for each $\alpha \in [0, 1) \cap \mathbb{Q}$, an effective method which produces a finite subset $\mathcal{S} \subset S(\Qbar)$ such that at least an $\alpha$-fraction of the points in $\mathcal{S}$ are Hodge generic.
\end{thm}
\noindent This can in particular be applied in the context of \autoref{thmLefschetzpen} to produce candidate hypersurfaces which are Hodge generic. We note that we are not aware of any methods for certifying that a given hypersurface in $\mathbb{P}^{n+1}$ is Hodge generic.

In a similar spirit to the above remark, we give another corollary which might be of independent interest. Recall that every smooth projective complex variety $Y$ has an associated \emph{Hodge diamond}, which is the collection of numbers $h^{p,q}= \dim H^p(Y,\Omega_Y^q)$; similarly there is the \emph{primitive Hodge diamond}, which is the collection of Hodge numbers appearing in the primitive parts in each degree. We say that a primitive Hodge diamond is \emph{irreducible} if the Hodge structure associated to each primitive cohomology group $H^j(Y,\Q)_{\operatorname{prim}}$ is irreducible.

\begin{cor}
Every irreducible primitive Hodge diamond realizable by a smooth projective complex variety is realizable by a smooth projective $\Qbar$-variety.
\end{cor}
\begin{proof}
   Let $Y$ be a complex variety with irreducible primitive Hodge diamond. By spreading out it can be seen as a fibre of a smooth projective family of varieties defined over $\Qbar$. Each member of this family will have the same primitive Hodge diamond, and irreducibility is guaranteed by \Cref{mainhodgegenthm} (see also \autoref{everythinggeneralizesrem} for a version that applies in all cohomological degrees). Indeed, the generic Mumford-Tate group of the family we built can be larger than that of $Y$, but it has to act irreducibly. The same applies then at every Hodge generic point.
\end{proof}

\subsection{Strategy: controlling period relations of bounded degree}
\label{strsec}

Our method for finding Hodge generic points ultimately reduces to the question of how to certifiably produce fibres of algebraic families whose periods do not satisfy exceptional $\Qbar$-algebraic relations of degree $\delta > 0$, where $\delta$ is some fixed integer. See indeed our discussion in \S\ref{periodconj}. 

For simplicity, we describe the strategy in the \emph{pure} case. Let $f : X \to S$ be a smooth projective family defined over a number field $K$, fix an integer $j \geq 0$, and set $\V = R^{j} f_{*} \mathbb{Q}$ to be the associated polarizable $\mathbb{Q}$VHS, which has rank $n \geq 0$. The task of finding a dense collection of Hodge generic $\Qbar$-points on $S$ for $\V$ reduces, by d\'evissage, to the case where the base has dimension one, and then to the case where $S \subset \mathbb{P}^1$ is a Zariski open subset. In this case, let $s$ be a rational parameter on $S$ that vanishes at a fixed point $s_{0} \in S(K)$. Let $\mathcal{H} = R^{j} f_{*} \Omega^{\bullet}_{X/S}$ be the de Rham realization of $\mathbb{V}$, and let $\nabla : \mathcal{H} \to \Omega^{1}_{S} \otimes \mathcal{H}$ be the Gauss-Manin connection. After shrinking $S$ we may assume that $\mathcal{H}$ is trivial with basis $\omega_{1}, \hdots, \omega_{n}$, and that $\Omega^{1}_{S}$ is trivialized by $ds$, with $s$ a parameter on $S$. Let $A(s)$ denote the associated connection matrix, with entries in the function field $K(s)$. 

Following the work of Andr\'e in \cite{zbMATH00041964}, the differential operator $d/ds - \mathbf{A}(s)$ is a $G$-operator; see \S\ref{Gopsec} and \S \ref{thmGoperator}. This means that any tuple $\mathbf{R}(s) = (G_{1}, \hdots, G_{n})^{t} \in K[[s]]^{n}$ of power series for which
\begin{equation}
\label{Gfunceq}
\frac{d}{ds} \mathbf{R}(s) = \mathbf{A}(s) \mathbf{R}(s)
\end{equation}
is a tuple of $G$-functions, cf. \Cref{defGfunc}.

\medskip

To use $G$-functions to study special moduli, one basic approach is as follows; for simplicity we assume $K = \mathbb{Q}$ and that $s_{0}$ is the point $\infty = [1 : 0] \in S$.\footnote{We note that in our main theorems we will even allow $s_0$ to be a point in $\mathbb{P}^1 \setminus S$, which means that one has to deal with possible quasi-unipotent monodromy. We ignore this here for simplicity.} Write $e_{1}, \hdots, e_{n}$ for the standard basis vectors of $\mathbb{Q}^{n}$. One starts by solving (\ref{Gfunceq}) with the initial conditions $\mathbf{R}(0) = e_{i}$ to obtain an $n \times n$ matrix $\mathbf{G}(s) = [\mathbf{G}_{ij}(s)]_{1 \leq i, j \leq n}$ of solutions. Viewing all data in the complex analytic setting, the matrix $\mathbf{G}(s)$ describes, in some analytic neighborhood $B$ of $s_{0}$, a varying change-of-basis matrix between the frame $\omega_{1}, \hdots, \omega_{n}$ and the unique flat frame that extends $\omega_{1,s_{0}}, \hdots, \omega_{n, s_{0}}$. Shrinking $B$ if necessary, we can consider a basis $b_{1,s_{0}}, \hdots, b_{n,s_{0}}$ for $\mathbb{V}_{s_{0}}$ and let $\mathbf{P}(0)$ be the de Rham-Betti period matrix at $s_{0}$ with respect to these bases. Then on $B$, one has a factorization
\begin{equation}
\label{Gfaceq}
\mathbf{G}(s) = \mathbf{P}(s) \mathbf{P}(0)^{-1} 
\end{equation}
where $\mathbf{P}(s)$ is the usual varying de Rham-Betti period matrix over $B$ (with respect to $\omega_{1}, \hdots, \omega_{n}$ and the flat extension $b_{1}, \hdots, b_{n}$ of $b_{1,s_{0}}, \hdots, b_{n,s_{0}}$). 

Then one wishes to control the points $s_{1} \in B \cap S(\mathbb{Q})$ where the fibre $X_{s_{1}}$ has extra Hodge cycles. A possible strategy, which was successfully implemented by Andr\'e (\cite[Ch. IX]{zbMATH00041964}, \cite[Thm. 1]{zbMATH00764346}) and is now well-known, is as follows:
\begin{itemize}
\item[(1)] Assume without loss of generality (by moving the point $s_0$) that the Hodge structure on $H^{j}(X_{s_{0}}, \mathbb{Q})$ carries extra Hodge tensors, so that the period matrix $\mathbf{P}(0)$ has lower transcendence degree than the matrix $\mathbf{P}(s)$ at a very general $s \in B$. Moreover choose $s_{1}$ so that it is close to $s_{0}$ at the unique archimedean place of $\mathbb{Q}$, and far from $s_{0}$ at each non-archimedean place.
\item[(2)] Consider a point $s_{1} \in B$ where the Hodge structure $H^{j}(X_{s_{1}}, \mathbb{Q})$ carries extra Hodge tensors, so that the transcendence degree of $\mathbf{P}(s_{1})$ also drops.
\item[(3)] Conclude from (\ref{Gfaceq}) that $\mathbf{G}(s_{1})$ has smaller transcendence degree than usual, and conclude from Bombieri's theorem \cite{zbMATH03721021} (or more precisely, \Cref{thm:bombieri}) that the heights of such $s_{1}$ are absolutely bounded. 
\end{itemize}
For our applications, this strategy has certain limitations. For a point $s \in B$, that the Hodge structure $H^{j}(X_{s}, \mathbb{Q})$ carries an extra Hodge tensor imposes only a weak constraint on the transcendence of $\mathbf{P}(s)$: although the Hodge conjecture would imply that $\mathbf{P}(s)$ lies inside a $\overline{\mathbb{Q}}$-algebraic torsor for the Mumford-Tate group of $H^{j}(X_{s}, \mathbb{Q})$, in the absence of the Hodge conjecture one only learns that the image of $\mathbf{P}(s)$ in a certain Hodge-theoretic flag variety is not $\overline{\mathbb{Q}}$-Zariski dense. This is true both when $s = s_{0}$ and $s = s_{1}$, but because this flag variety constraint is so weak, one cannot deduce the existence of any additional $\overline{\mathbb{Q}}$-algebraic relations on $\mathbf{G}(s_{1})$. In fact, even if one knew the Hodge conjecture, one would only be able to deduce the presence of additional $\Qbar$-algebraic relations when the Mumford-Tate groups at $s_0$ and $s_1$ are sufficiently small (roughly half the dimension of the generic Mumford-Tate group), and one could not control all special points on $S$ in this way.

\medskip

To avoid this issue, we adopt the following strategy that will culminate in \Cref{thmwithproof}. Fix an integer $N \gg 0$. We can replace the variation of Hodge structures $\mathbb{V}$ with the variation of Hodge structure $\mathbb{V}^{(N)}$ on the $N$-fold self-product 
\begin{equation}
\label{SNprod}
S^{N} := S \times \cdots \times S\hspace{4em} \mathbb{V}^{(N)} = \textrm{pr}^{*}_{1} \mathbb{V} \oplus \cdots \oplus \textrm{pr}^{*}_{N} \mathbb{V} .
\end{equation}
where $\textrm{pr}_{i} : S^{N} \to S$ is the obvious projection. Consider a generic rational curve $C \subset S^{N}$ passing through $0 := (s_{0}, \hdots, s_{0})$, and let $p_{i}$ be the restriction of $\textrm{pr}_{i}$ to $C$, which is a birational map. Fix a uniformizing parameter $c$ on $C$ at $0$, and consider the associated open embedding $\iota : C \setminus \{ 0 \} \hookrightarrow \mathbb{A}^1$ given by $1/c$. We can then consider the $G$-function system $\mathbf{G}$ with components 
\begin{equation}
\label{Nfoldsystem}
\underbrace{\mathbf{G}_1(c)}_{(\mathbf{P} \circ p_{1})(c) \mathbf{P}(0)^{-1}}, \hdots, \underbrace{\mathbf{G}_N(c)}_{(\mathbf{P} \circ p_{N})(c) \mathbf{P}(0)^{-1}} .
\end{equation}
To produce Hodge generic points on $S$, it then suffices to consider the counter-factual situation where for some $c_{1} := (s^{1}, \hdots, s^{N})$ in $C(\mathbb{Q}) \cap \iota^{-1}(\mathbb{A}^1(\mathbb{Z}))$ \emph{all} entries of the tuple $c_{1}$ are not Hodge generic. Then instead of a mere one relation coming from the flag variety associated to $\mathbb{V}^{(N)}_{c_{1}}$ we will have at least $N$ of them, and if $N$ is larger than the transcendence degree of $\mathbf{P}(0)$ we can eliminate $\mathbf{P}(0)^{-1}$ to produce a non-trivial algebraic relation on the tuple $\mathbf{G}{_1}(c_{1}), \hdots, \mathbf{G}_{N}(c_{1})$, and then, using that $\iota(c_{1}) \in \mathbb{A}^1(\mathbb{Z})$, apply Bombieri's Theorem to control the height of $c_{1}$ for which such a relation appears.

We note that this strategy only works if the ``extra'' $\Q$-relations on $\mathbf{G}_{i}(c_{1})$ that appear are of uniformly bounded degree, and to show this is the case one needs non-trivial input from Hodge theory (in particular, here it becomes crucial that the Hodge structures we consider are polarizable). 

\medskip

Our approach takes inspiration from the Zilber-Pink conjecture for $\mathbb{Z}$VHS, as formulated in \cite{2021arXiv210708838B} for the pure case and \cite{2024arXiv240616628B} for the mixed one. Roughly, this conjecture predicts that how often a special point occurs inside $S(\mathbb{C})$ should be predicted by computing the codimension of (the image of) $S$ inside a space parameterizing Hodge-theoretic data associated to $\mathbb{V}$. Replacing $S$ with a (Hodge generic) curve $C$ in $S^{N}$ makes this codimension larger, and should therefore make points with all entries not Hodge generic less likely, and then (potentially) easier to control. The use of $G$-functions to study special moduli was initiated by Andr\'e in his book \cite{zbMATH00041964}, and recently there has been renewed interest in applying $G$-function-based methods in connection with the Zilber-Pink conjecture; we refer to \cite{2025arXiv250109867A} for a survey. Interestingly, the idea of considering a curve in a self-product of the base appears also in the work of Dèbes and Zannier \cite{zbMATH01096013} (especially Theorem 1 and 2 in \emph{op. cit.}), where, inspired by certain generalizations of Hilbert’s irreducibility theorem, they study relations among special values of $G$-functions.

\subsection{Plan of the paper}
In \Cref{section2} we describe in detail the motivic conjectures which appeared at the beginning of \S \ref{mmr} and give more general results in the spirit of \autoref{mainthmhyper}.

\Cref{sec3} and  \Cref{Gresult} are devoted to $G$-operators and $G$-functions. In  the former, we introduce the necessary background and in the latter state and prove the core result of the paper: \Cref{thmwithproof}.

In \Cref{sec5} and \Cref{sec6} Hodge theory enters the picture. After recalling the necessary background about the Hodge locus, in \S \ref{sec6} we prove the abundance of algebraic Hodge generic points and deduce from that the various announced applications. Finally in the short appendix \ref{app} we compare our results with an observation of Siegel.

\subsection*{Acknowledgments}
We thank Fran\c{c}ois Charles and Peter Jossen for comments on a first version of the paper and their interest in our work.

This work was done while the authors were at the Institute for Advanced Study in Princeton and they would like to thank the institute for its hospitality and for providing excellent working conditions. The first G.B. and D. U. are also grateful for the support of the Ambrose Monell Foundation as well as the Giorgio and Elena Petronio Fellowship Fund, and the second G.B. of the Marvin V. and Beverly J. Mielke Endowed Fund and the Infosys Member Fund.

The first G.B. was partially supported by the grant ANR-HoLoDiRibey of the Agence Nationale de la Recherche. The second G.B.
was supported by the European Union (ERC, SharpOS, 101087910), and by the Israel Science Foundation (grant No. 2067/23).

\section{Motivic conjectures: background and general results}\label{section2}

In addition to the results announced in the previous section and the general $G$-function \Cref{thmwithproof}, we wish to record here two more new statements, \Cref{galperiodconj} and \Cref{mainthmonMT}, that generalize \Cref{mainthmhyper} and will be proven in \S \ref{sec6}. For simplicity, we focus only on the ``pure case'' and return to the general ``mixed case'' in \S \ref{sec5} and \S \ref{sec6}.

\subsection{Hodge conjecture}

A fundamental question in algebraic geometry is to understand, given a smooth projective family $f : X \to S$, which fibres $X_{s}$ of $f$ carry more algebraic cycles than the geometric generic fibre $X_{\overline{\eta}}$ of $f$. Here by ``carry more algebraic cycles'' we mean that for some Weil cohomology theory $H^{\bullet}(-)$, and some integer $k$, the cohomology $H^{\bullet}(X^{k}_{s})$ (where $X^{k}$ denotes the $k$-fold self-fibre product) admits an algebraic cycle class which does not arise by specializing an algebraic cycle on $X_{\overline{\eta}}$; see also the introduction of \cite{2021arXiv210708838B}.

Several central conjectures in mathematics give precise criteria on when such extra cycles should exist. The first of these is the Hodge conjecture, which works with the Betti cohomology.

\begin{conj}(Hodge conjecture)
\label{Hgconj}
Let $Y$ be a smooth projective algebraic variety over $\mathbb{C}$, let $j \geq 0$ be an integer, and suppose that $c \in H^{2j}_{B}(Y, \mathbb{Q}(j))$ is a class which lies inside the middle Hodge summand $H^{j,j} \subset H^{2j}_{B}(Y, \mathbb{C})$. Then $c$ is in the $\mathbb{Q}$-span of the algebraic cycle classes.
\end{conj}

\noindent Classes $c$ as in \autoref{Hgconj} are called \emph{Hodge classes}. More generally, given a rational Hodge structure $V$, a \emph{Hodge tensor} is an element $v \in V^{\otimes a }\otimes V^{\vee \otimes b} \otimes \Q(c)$ which is of type $(0,0)$ for the Hodge decomposition on $V^{\otimes a }\otimes V^{\vee \otimes b} \otimes \Q(c)$. The Hodge conjecture says that Hodge tensors are $\mathbb{Q}$-linear combinations of algebraic cycle classes on a certain self-product of $Y$, twisted appropriately by the Tate motive. This also explains the notation $\mathrm{H}(Y,j)$ appearing in the previous section. 

\begin{defn}
\label{MTdef}
Let $V$ be a $\mathbb{Q}$-Hodge structure, viewed as a representation $h : \mathbb{S} \to \GL(V)_{\mathbb{R}}$ of the Deligne torus $\mathbb{S} = \operatorname{Res}_{\mathbb{C}/\mathbb{R}} \mathbb{G}_{m,\mathbb{C}}$. Then the \emph{Mumford-Tate group} $\MT(V)$ of $V$ is the $\mathbb{Q}$-Zariski closure of the image of $h$. If $V$ is polarized, then this is the fixator of all Hodge tensors.
\end{defn}

Let $\V = R^{j} f_{*} \mathbb{Q}$ be the $\mathbb{Q}$VHS appearing in degree $j$ associated to $f$. As mentioned in \S\ref{Hodgegenptssec}, a non-Hodge generic point $s \in S(\mathbb{C})$ is a point where the Mumford-Tate group of $\mathbb{V}_{s}$ is smaller than for a very general $s$. Since every algebraic cycle class is a Hodge class, and Hodge classes constrain the size of the Mumford-Tate group, a Hodge generic point is necessarily a point at which there are no extra algebraic cycles. It is moreover known by the Cattani-Deligne-Kaplan theorem \cite{CDK} that the condition of carrying an extra Hodge class gives rise to a countable union of closed algebraic subschemes of $S$, called the \emph{Hodge locus}. But because a Hodge structure is a transcendental object constructed from periods, it is difficult to certify for any given point $s \in S(\mathbb{C})$ that $s$ does not lie in the Hodge locus. In particular if $f$ is defined over $\Qbar$, one would like to know that there are points in $S(\Qbar)$ that do not lie in the Hodge locus. \Cref{mainhodgegenthm} shows this is the case.

\subsection{Period conjecture}\label{periodconj}

We give some partial results on the period conjecture of Grothendieck, see e.g. \cite[IX.2.2]{zbMATH00041964} and \cite[Sec. 7.5]{zbMATH02128361} for background. 

\begin{conj}(Grothendieck Period Conjecture)
\label{Gperconj}
Let $Y$ be a variety over a number field $K \subset \mathbb{C}$ and $j$ a natural number. Then the transcendence degree (over $\Q$) of the field generated by the period matrix for the Betti-de Rham period morphism $H^j(Y_{\mathbb{C}}, \mathbb{Q}) \otimes \mathbb{C} \xrightarrow{\sim} H^j_{\operatorname{dR}}(Y) \otimes_{K} \mathbb{C}$ is equal to the dimension of the Mumford--Tate group of the Hodge structure on $H^j(Y_{\mathbb{C}}, \mathbb{Q})$.
\end{conj}

We wish to show that as $Y$ varies in a family of algebraic varieties, the Grothendieck period conjecture is ``often'' true. However, we will only be able to establish something weaker, which controls the degree over $K$ over which possible exceptional relations are defined. 

Let $f : X \to S$ be a smooth projective family defined over $\Qbar$, and let $j \geq 0$ be an integer. For simplicity let us shrink $S$ so that $R^j f_{*} \Omega^{\bullet}_{X/S}$ is trivial. Consider the Betti-de Rham comparison 
\begin{equation}
    \mathbf{P} : \underbrace{R^j f_{*} \mathbb{Q}}_{=: \mathbb{V}} \otimes \mathcal{O}_{S^{\an}} \simeq [\underbrace{R^j f_{*} \Omega^{\bullet}_{X/S}}_{=: \mathcal{H}}]^{\an}.
\end{equation}
Fixing bases for $\mathcal{H}$ and $\mathbb{V}$ (the latter defined only up to monodromy), one obtains a multi-valued varying period matrix $\mathbf{P} : S^{\an} \to \GL_{n}(\mathbb{C})$, where $n$ is the rank of the cohomology in degree $j$. Let $B_{\mathbf{P}} \subset S \times \GL_{n}$ be the $\Qbar$-Zariski closure of the graph of $\mathbf{P}$. Then the Grothendieck period conjecture, together with the Hodge conjecture for self-products of a very general fibre $X_{s}$ of $f$, implies the following:

\begin{conj}
\label{perconj}
For all points $s \in S(\Qbar)$ outside the Hodge locus, the $\Qbar$-Zariski closure of $\mathbf{P}(s)$ is an irreducible component of $B_{\mathbf{P},s}$. 
\end{conj}
\noindent This consequence of the period conjecture can be broken up into a series of weaker statements indexed by an integer $\delta$, as follows:
\begin{conj}\label{degdeltaperconj}
For all points $s \in S(\Qbar)$ outside the Hodge locus, and for a fixed integer $\delta$, the ideal defining the $\Qbar$-Zariski closure of $\mathbf{P}(s)$ has no elements of degree at most $\delta$ defined over $\kappa(s)$ which do not vanish on some irreducible component of $B_{\mathbf{P},s}$ passing through $\mathbf{P}(s)$. As shorthand, we say that the $\Qbar$-Zariski closure of $\mathbf{P}(s)$ carries ``no nontrivial relations of degree $\delta$ over $\kappa(s)$''.
\end{conj}
\noindent Correspondingly, one also has a locus
\[ S_{P,j,\delta} := \{ s \in S(\Qbar) : \overline{\mathbf{P}(s)}^{\Qbar-\operatorname{Zar}}\textrm{ has no nontrivial relations of degree }\delta\textrm{ over }\kappa(s) \} .  \]

\noindent We note that although our construction of $\mathbf{P}$ required assuming $\mathcal{H}$ is trivial, the locus $S_{P,j,\delta}$ is independent of the basis for $\mathcal{H}$ chosen, and so makes sense even if $\mathcal{H}$ is not trivial over $S$. Likewise the statements of \autoref{perconj} and \autoref{degdeltaperconj} do not depend on this choice of basis and hence also globalize.

It turns out that our general $G$-function-based method shows that points where \Cref{degdeltaperconj} holds are abundant in $S$. Using it, we will show the following.

\begin{thm}\label{galperiodconj}
For each positive integer $\delta$, $S_{P,j,\delta}$ is analytically dense in $S(\mathbb{C})$.
\end{thm}

\autoref{galperiodconj} does not directly give any results on the period conjecture; for instance, if the family $f$ is a constant family with fixed fibre $Y$, then \autoref{galperiodconj} says nothing about the periods of $Y$. Under the favourable assumptions discussed in the next section, however, \autoref{galperiodconj} gives results on the period conjecture itself, allowing us to prove an ``up to degree $\delta$'' version in \autoref{mainthmonperiod}. We note that \autoref{mainthmonperiod} appears to be already new for the following one parameter families of hyperelliptic curves (even in genus one):

\begin{cor}
\label{hypercor}
Let $f \in \mathbb{Q}[x]$ be a square free polynomial  of degree $2g$ with $g \geq 1$, and consider the family
\[ C_{t} : y^2 = f(x)(x-t) \]
of hyperelliptic curves. Then for each $\delta > 0$ consider the subset $\mathbb{N}_{\delta} \subset \mathbb{N}$ defined by the property that the periods $\int_{\gamma_{j}} \frac{x^{i} dx}{y}$ of $C_{t}$ satisfy the period conjecture in degree $\delta$ for $t \in \mathbb{N}_{\delta}$. Then $\mathbb{N}_{\delta}$ has natural density one. Moreover, there is an effective method which produces a subset $\mathcal{S} \subset \mathbb{N}$ such that ``99 percent'' (in the sense of \Cref{effectivethmintro}) of the $C_{t}$ for $t \in \mathcal{S}$ satisfy the period conjecture up to degree $\delta$.
\end{cor}

\subsection{Tate and Mumford-Tate conjectures}
When $f$ is defined over a number field $K \subset \Qbar \subset \mathbb{C}$, one can also study the locus where extra algebraic cycles appear using the Tate conjecture.

\begin{conj}(Tate conjecture)
\label{Tateconj}
Let $Y$ be a smooth projective algebraic variety defined over a number field $K$, let $j, \ell \geq 0$ be integers with $\ell$ prime and suppose that $c \in H^{2j}_{\'et}(Y_{\overline{K}}, \mathbb{Q}_{\ell}(j))$ is a class invariant under a finite index subgroup of $\textrm{Gal}(\overline{K}/K)$. Then $c$ is in the $\mathbb{Q}_{\ell}$-span of the algebraic classes.
\end{conj}

\noindent Classes $c$ as in \autoref{Tateconj} are called \emph{Tate classes} (of cohomological degree $j$). Similarly there is also a definition of \emph{Tate tensors}. The Tate conjecture says that Tate classes are linear combinations of algebraic cycle classes.

Since both Hodge and Tate classes should correspond to algebraic cycle classes, the two notions of genericity should be the same. This is the subject of the wide open Mumford-Tate conjecture.

\begin{conj}($\ell$-Mumford-Tate)
\label{MTconj}
Let $Y$ be a smooth projective algebraic variety defined over a number field $L \subset \mathbb{C}$, and let $\ell$ be a prime. Then for each pair of integers $(j, k)$ the Artin comparison map
\[ H^{2j}_{B}(Y^{k}_{\mathbb{C}}, \mathbb{Q}) \otimes_{\mathbb{Q}} \mathbb{Q}_{\ell}(j) \xrightarrow{\sim} H^{2j}_{\textrm{\'et}}(Y^{k}_{\overline{L}}, \mathbb{Q}_{\ell}(j)) \]
identifies the $\mathbb{Q}_{\ell}$-vector spaces spanned by Hodge and Tate classes.
\end{conj}

\noindent If the Mumford-Tate conjecture were known, one could use it to produce Hodge generic points from $\ell$-Tate generic points, and then deduce \autoref{mainhodgegenthm} from Serre's version of Hilbert irreducibility theorem, cf. \S \ref{relworksec}. But, like both the Hodge and Tate conjectures, the Mumford-Tate conjecture is wide open. However, a posteriori it turns out that we can combine the two density results to say something about \Cref{MTconj} more general than \Cref{mainthmonperiod}. First, following Deligne \cite[Sec. 2]{zbMATH03853242}, we recall what it means for a Hodge class to be $K$-absolute Hodge, with $K \subset \mathbb{C}$ a number field. Let $Y$ be a smooth projective complex variety and let
$v\in H_B^j(Y,\mathbb Q(c))$ be a rational Hodge class. Write
$v_{\mathrm{dR}}$ and $v_\ell$ for its de Rham and $\ell$-adic
realizations. We call $v$ \emph{$K$-absolute Hodge} if, for every
$\sigma\in\operatorname{Aut}(\mathbb C/K)$, there exists a rational
Hodge class $v^\sigma\in H_B^j(Y^\sigma,\mathbb Q(c))$ such that
\[
 (v^\sigma)_{\mathrm{dR}}=\sigma_*v_{\mathrm{dR}},
 \qquad
 (v^\sigma)_\ell=\sigma_*v_\ell
 \quad\text{for every prime }\ell.
\]
Here the $\ell$-adic realizations lie in
$H^j_{\mathrm{\acute et}}(Y,\mathbb Q_\ell(c))$ and
$H^j_{\mathrm{\acute et}}(Y^\sigma,\mathbb Q_\ell(c))$, respectively;
all comparison maps include the indicated Tate twists. The same
definition applies to tensor constructions, duals and Tate twists.
In what follows, \emph{absolute Hodge} means $K$-absolute Hodge.

To explain the general version of \Cref{mainthmhyper}, we introduce some notation. Given a smooth projective family $f : X \to S$ defined over $K$ and a fixed integer $j$, let $\V$ be the variation of Hodge structures associated to the primitive cohomology $H^j(X_s, \mathbb{Q})_{\mathrm{prim}}$. Associated to $\V$, we consider the following groups:
\begin{enumerate}
\item[(1)] $\mathbf{G}=\mathbf{G}(f,j)$, the generic Mumford--Tate group of the $\mathbb{Q}$VHS, i.e., the stabilizer of all Hodge tensors at a very general point $s \in S(\mathbb{C})$;
\item[(2)] $\mathbf{G}_{\mathrm{abs}}=\mathbf{G}_{\mathrm{abs}}(f,j)$, the generic \emph{absolute Mumford--Tate group}, i.e.\ the pointwise stabilizer of all absolute Hodge tensors at a very general point $s \in S(\mathbb{C})$; and
\item[(3)] $\mathbf{M}=\mathbf{M}(f,j)$, the identity component of the (Zariski closure of the) monodromy representation of $\pi_1(S)$ acting on $\V$ (cf. \Cref{mondef}).
\end{enumerate}
Note that (2) is well-defined: a tensor $t$ which is absolute Hodge for $H^j(X_s, \mathbb{Q})_{\mathrm{prim}}$, with $s \in S(\mathbb{C})$ outside the tensorial Hodge locus of $\mathbb{V}$, in fact extends to define an absolute Hodge tensor above every point in $S$, up to a finite orbit. In particular as an abstract group $\mathbf{G}_{\textrm{abs}}$ does not depend on the very general point $s$ chosen. See \cite[Prop. 5.6]{zbMATH06492665}. If $Y$ is defined over a number field $L\supseteq K$, the $\ell$-adic
realization of a $K$-absolute Hodge tensor has finite
$\operatorname{Gal}(\overline L/L)$-orbit, as in \cite[Prop. 2.9 (b)]{zbMATH03853242}. Indeed, compatibility for automorphisms fixing
$\overline{\mathbb Q}$ first implies that its de Rham realization is
defined over $\overline{\mathbb Q}$. After a finite extension $L'/L$
this realization is fixed; compatibility then implies that every
$\ell$-adic realization is fixed by
$\operatorname{Gal}(\overline L/L')$. Consequently these tensors are
fixed by the identity component of the algebraic Galois group.

Recall that $\mathbf{G} \subset \mathbf{G}_{\mathrm{abs}}$, and, unless the weight of the variation of Hodge structure is zero, both contain a canonical copy of $\mathbb{G}_m$ corresponding to the homotheties. We denote by $\mathbf{G}'$ and $\mathbf{G}_{\mathrm{abs}}'$ the corresponding quotients by this $\mathbb{G}_m$ (when the weight is zero this notation is superfluous). Note that $\mathbf{M}$ is semisimple connected and $\mathbf{M} \subset \mathbf{G}$. Write $\mathbf{M}' = \mathbf{M}'(f,j)$ for the image of $\mathbf{M}$ in $\mathbf{G}$.

In the remainder of the section, we present results under the following assumption: 
\begin{equation}\label{bigmon}
  \mathbf{M}'(f,j) = \mathbf{G}'(f,j) = \mathbf{G}_{\mathrm{abs}}'(f,j)  .
\end{equation}
Condition \eqref{bigmon} is satisfied, for example, when $f:X \to S$ has large monodromy\footnote{A family $f:X \to S$ is said to have \emph{large monodromy} (in cohomological degree $j$) if the algebraic monodromy associated with the local system $[R^{j} f_* \Q]_{\operatorname{prim}}$ with its natural polarization is just the $\mathbb{Q}$-group of polarization preserving automorphisms $\operatorname{Aut}(V,Q)$; cf. the notation of \S\ref{vhsandpermapsec}.}. This is the case for the universal family of degree $d$ smooth hypersurfaces in any projective space, as well as the universal families of smooth complete intersections. Moreover, the equality $\mathbf{G}'(f,j) = \mathbf{G}_{\mathrm{abs}}'(f,j)$ always holds if $f$ is an abelian family, so it suffices to check the first inequality in (\ref{bigmon}): in particular, any non-isotrivial abelian family for which $\mathbf{G}'$ is $\mathbb{Q}$-simple will satisfy \eqref{bigmon}. Other cases where such an assumption is satisfied are given by Mustafin's work \cite{zbMATH04029733}; see also \cite[Ch. IX 3.2]{zbMATH00041964}.

\begin{thm}\label{mainthmonMT}
Let $f : X \to S$ be a smooth projective family defined over a number field $K \subset \Qbar \subset \mathbb{C}$ and $j$ a natural number. If $[R^jf_*\Q]_{\operatorname{prim}}$ satisfies \eqref{bigmon} then, for all $\ell$, the locus
$$S_{\operatorname{MT}, \ell,j}\subset S(\Qbar)$$
is analytically dense in $S(\C)$.
\end{thm}

\subsection{Related work}
\label{relworksec}

The most common approach to the existence of Hodge generic points is via the ``easy direction'' of the Mumford-Tate conjecture: roughly, the statement that Hodge cycles associated to a cohomology group $H^{j}(X_{s}, \mathbb{Q})$ above a point $s \in S(K)$ should constrain the image of corresponding $\ell$-adic Galois representations $\rho_{s} : \textrm{Gal}(\overline{K}/K) \to \GL(H^{j}(X_{s,\overline{K}}, \mathbb{Q}_{\ell}))$. This is automatic under the Hodge conjecture, as well as weaker versions of the Hodge conjecture that assert such cycles have appropriate motivic properties. For instance, Voisin in \cite[Thm. 1.11]{zbMATH06492665} explains how the existence of many Hodge generic $S(\Qbar)$-points can be proven using the conjecture that Hodge cycles are absolute Hodge. In a similar vein, Andr\'e in \cite[Thm 5.2(3)]{zbMATH01019346} gives a similar statement for points in $S(\Qbar)$ admitting extra cycles that are motivated in an appropriate sense, and \cite[Thm. 1]{zbMATH03927068} gives such arguments in the explicit case of complete intersections and where the Hodge cycles one wishes to avoid come from line bundles.

The basic principle of all such arguments can be found in \cite{zbMATH03927068}. The idea is to assemble the representations $\rho_{s}$ into a family by considering the associated representation $\rho : \pi^{\'et}_{1}(S,s_{0}) \to \GL(H^{j}(X_{s_{0},\overline{K}}, \mathbb{Q}_{\ell}))$ of the \'etale fundamental group at some base-point $s_{0} \in S(K)$, and then study the set of $s$ for which the restriction $\rho_{s}$ maps into a strict subgroup of the image of $\rho$. A result of Serre \cite[\S10.6]{zbMATH00042767} (cf. the introduction to \cite{zbMATH06506624} and \cite[Fact 3.3.1.1]{zbMATH06657570}) can be used to show that such $s$ lie in a \emph{thin set}, defined as in \cite[\S9.1]{zbMATH00042767}. When $S$ has many $K$-points, one can show that the complement of such a thin set is analytically dense, hence obtain the corresponding Hodge-theoretic density statement under Hodge or absolute Hodge-like assumptions.

We briefly mention that, in the case where $f$ is an abelian family, Masser~\cite{zbMATH01181633} proved that the rank of the endomorphism ring does not jump at almost all $\overline{\Q}$-points using transcendental methods. This result builds on difficult work by Masser-W\"ustholz on endomorphism ring estimates for abelian varieties. We will not have such tools available to us.

In a related direction, the $p$-adic density of the Noether--Lefschetz locus is investigated in \cite{zbMATH06073742}. Using crystalline cohomology, the authors of \cite{zbMATH06073742} strengthen Terasoma’s theorems by proving that the Noether-Lefschetz locus is nowhere $p$-adically dense, under adequate assumptions on $p$. We refer also to the surveys \cite{zbMATH07243792}, and \cite{zbMATH07745042} for results on related topics. In general, the mixed case seems to be less investigated. 

\medskip

The period conjecture, \Cref{Gperconj}, is known for the cohomology of projective spaces and elliptic curves with complex multiplication; the case of projective spaces is the transcendence of $\pi$, and the result for elliptic curves is due to \cite{zbMATH03670527}. For other special cases, we refer also to \cite{zbMATH06576591}.

André has also used the $G$-function method to address certain cases of the period conjecture up to degree $\delta$, however, and this result is perhaps the closest to our \Cref{galperiodconj} and \Cref{mainthmonperiod}. Indeed, in \cite[Ch. IX (page 193)]{zbMATH00041964} Andr\'e considered families $X \to S$ of abelian schemes with (completely) multiplicative reduction in $s_0$, and projective smooth morphisms satisfying a certain strong degeneration condition in $s_0$. (See page 176 for the definition of \emph{multiplicative reduction}, as well as Mustafin's work \cite{zbMATH04029733}). He then proved that \textbf{all} points of $S(K)$ of sufficiently large height and which are sufficiently close to $s_0$ satisfy the following condition: every polynomial relation of degree $\le \delta$ with coefficients in $K$ among the values at $s$ of the locally invariant periods comes from specializations at $s$ of relative Hodge cycles. The main differences are that our approach works with every family (i.e. without any degeneration assumptions), considers the entire period matrix, and does not require working around a point $s_0$ where the family degenerates. However, our conclusion is weaker and only shows that ``99 percent'' of the points we produce are Hodge generic.

As for the Mumford-Tate conjecture, to the best of our knowledge, little is known outside of motives that can be related to abelian varieties, where various partial results are known; see, for instance, the survey \cite{Moonen} of Moonen. 
For abelian varieties $A$ themselves, the conjecture is still open (even in dimension $4$), but is known in some cases, including the case where $A$ is geometrically simple of prime dimension \cite{zbMATH03941671}.

\section{$G$-functions and height bounds}\label{sec3}

\subsection{$G$-operators and $G$-functions at infinity}
\label{Gopsec}

By a \emph{differential operator} $L$ we mean an element $L \in \mathcal{D} := \Qbar[z][\partial]$, where $\Qbar[z][\partial]$ is the standard Weyl algebra, with $\partial = \frac{d}{dz}$. For this section, we largely follow \cite{zbMATH00041964}, \cite[Sec. 1]{zbMATH07441738} for a concise and recent reference, and \cite{zbMATH01860838} for background on linear differential equations. We start by recalling the definition of $G$-function, a concept introduced by Siegel \cite{zbMATH02563202}:

\begin{defn}\label{defGfunc}

A $G$-\emph{function} is a formal power series with algebraic coefficients
\[
G(z)=\sum_{n=0}^{\infty} a_n z^n \in \overline{\mathbb{Q}}[[z]] ,
\]
such that $G$ is annihilated by a non-zero differential operator
\(L \in \mathcal{D}\)
and satisfies the following growth conditions: for each \(n \ge 1\),
writing \(d_n \ge 1\) for the smallest integer such that
\(d_n a_1, \ldots, d_n a_n\) are algebraic integers, there exists
a real number \(C>0\) such that \(|\sigma(a_n)| \le C^n\) and \(d_n \le C^n\)
hold for all \(\sigma \in \mathrm{Gal}(\overline{\mathbb{Q}}/\mathbb{Q})\)
and all \(n \ge 1\).
\end{defn}
\noindent Note that the set $\{ a_n \}_{n \geq 0}$ generates a finite extension of $\mathbb{Q}$.

By a \emph{differential algebra} $\mathcal{A}$ we mean an algebra $\mathcal{A}$ which is a module over $\mathcal{D}$. Similarly, a \emph{differential module} is a module over $\Qbar[z]$ equipped with the action of a differential operator. Given a differential operator $L$, we obtain a natural differential module $M := \mathcal{D} / \mathcal{D} L$.

We consider the differential algebra
\begin{equation}
\label{Meq}
    \mathcal{M} = \mathbb{C}\{\!\{z\}\!\}\big[(z^a)_{a\in\mathbb{C}}, \log(z)\big],
\end{equation}
whose elements are expressions
\begin{equation}
\label{Peq}
    P(z) = \sum_{i=1}^m z^{a_i}\log(z)^{b_i}G_i(z),
\end{equation}
where $a_i\in\mathbb{C}$, $b_i\in\mathbb{Z}_{\ge 0}$, $m$ is a non-negative integer, and where the $G_i$ are convergent Laurent series.
We consider the differential subalgebra $\mathcal{G} \subset \mathcal{M}$ which consists of expressions of the form
\begin{equation}
\label{Geq}
P(z) = \sum_{i=1}^m c_i z^{a_i}\log(z)^{b_i}G_i(z),
\end{equation}
with $a_i\in\mathbb{Q}$, $b_i\ge 0$, $c_i\in\mathbb{C}$, $m$ is a non-negative integer, and where the $G_i$ are $G$-functions (cf. \Cref{defGfunc}). In this framework, elements of $\mathcal{G}$ include formal expressions corresponding to multi-valued functions algebraic over $\overline{\mathbb{Q}}(z)$, and the above representation is (essentially) unique if one further requires that $0 \leq a_{i} < 1$ for all $i$.

Following \cite[\S2.2]{zbMATH01485629} we define:

\begin{defn}
Suppose $M$ is a differential module, and $\mathcal{A}$ is a differential algebra. A \emph{solution} of $M$ in $\mathcal{A}$ is a morphism $s : M \to \mathcal{A}$ of differential modules.
\end{defn}

\noindent In particular if $L$ is a differential operator, a \emph{solution to $L$ valued in} $\mathcal{A}$ means a map $s : \mathcal{D} / \mathcal{D} L \to \mathcal{A}$. 

\begin{defn}
    The \emph{rank} of $L$ is the highest power of $\partial$ appearing in the expression defining the operator. 
\end{defn}

Another invariant associated to $L$ is the dimension of a maximal linearly independent set of solutions valued in $\mathcal{M}$. The operator $L$ is said to have \emph{regular singularities} if this is equal to the rank of $L$. Such a maximal linearly independent set of solutions is called a \emph{basis}.
\begin{defn}
\label{Gopdef}
A differential operator $L \in \mathcal{D} := \Qbar[z][\partial]$ is said to be a \emph{G-operator} if $L$ admits a basis of solutions valued in  $\mathcal{G}$, and where the $c_{i}$ lie in a number field.
\end{defn}

\noindent The above definition is equivalent to other standard definitions of $G$-operators found in the literature, cf. \cite[Sec. 1.6]{zbMATH07441738}. Note that \cite[Sec. 1.6]{zbMATH07441738} does not impose the additional condition that $c_{i} \in \Qbar$, but it is equivalent. 

We will need the following fundamental results on $G$-operators obtained by combining the work of André, Chudnovsky, and Katz. A proof can be found in \cite[Thm. 1.8]{zbMATH07441738}. See also the references therein, which include \cite{zbMATH03893285, zbMATH03985372}, for a complete history of this result. For the notions of products and duals of differential operators we refer to \emph{op. cit.} and \cite[Ch. 2]{zbMATH01860838}.

\begin{thm}[André, Chudnovsky, Katz]\label{ACK}
$G$-operators satisfy the following properties:
\begin{enumerate}
\item Every $G$-function, and more generally every element of $\mathcal{G}$, is annihilated by a $G$-operator.
\item Products and duals of $G$-operators are $G$-operators, and every left or right factor of a $G$-operator is a $G$-operator.
\item $G$-operators have regular singularities on $\mathbb{P}^1$, all with rational local exponents.
\item If $L$ is a $G$-operator, then so is $[h]^*L$ for any non-constant rational function $h \in \overline{\mathbb{Q}}(z)$.
\end{enumerate}
\end{thm}

Note that if one applies the Fuchs criterion and Frobenius method (cf. \cite[Sec. 1.1]{zbMATH07441738}, \cite[Ch. III, \S1]{zbMATH00041964}) to a differential operator over $K$ with regular singularities, one obtains a basis of solutions that naturally has coefficients in a finite extension $K'$ of $K$, with the extension $K'$ being the splitting field of the indicial equation associated to the operator. In the case of rational local exponents the indicial equation has rational roots and so splits over $K$, so one has $K' = K$. Thanks to \autoref{ACK}(3), this applies when the differential operator is a $G$-operator.

Each differential module $M$ has a generic fibre $\Qbar(z) \otimes_{\Qbar[z]} M$ which is naturally a $\Qbar(z)[\partial]$-module. Suppose that $M = \mathcal{D} / \mathcal{D} L$ with $L$ a $G$-operator, and we consider a basis $s_{1}, \hdots, s_{n} : M \to \mathcal{G}$ of solutions, with $n$ the rank of $L$. If we choose additionally a $\Qbar(z)$-basis $m_{1}, \hdots, m_{n}$ of the generic fibre of $M$, we obtain a matrix 
\[ \mathbf{P} = [ s_j(m_i) ]_{1 \leq i, j \leq n} \]
which we call a fundamental matrix of solutions of $L$. If $\mathbf{G}_{\textrm{log}}$ is the fundamental matrix of a $\Qbar$-solution, then any other fundamental matrix $\mathbf{P}$ is of the form 
\begin{equation}\label{Qacts}
    \mathbf{P} = \mathbf{G}_{\textrm{log}} \cdot  Q
\end{equation}
where $Q \in \GL_{n}(\mathbb{C})$ is a constant matrix.

\begin{rem}
We emphasize that the entries of $\mathbf{P}$ need not be $G$-functions themselves, but are instead linear combinations of $G$-functions with coefficients of the form $c z^{a} \log(z)^{b}$, with $c$ a complex number, $a$ rational, and $b$ an integer. More generally, such solutions allow one to relate the theory of $G$-functions to the theory of periods, where one considers analytic functions which have complex coefficients and may have non-trivial local monodromy.
\end{rem}

Since it will appear in the sequel, we recall the notion of \emph{tensor product}. Let $L_1, L_2$ be differential operators with solution spaces $\operatorname{Sol}_{L_1}$ and $\operatorname{Sol}_{L_2}$. There is an operator $L_1\otimes L_2$ whose solutions are spanned by 
\begin{displaymath}
    \{y_1 y_2 : y_1\in \operatorname{Sol}_{L_1}, y_2\in \operatorname{Sol}_{L_2} \}.
\end{displaymath}
It is denoted as a tensor product because of its relation with the tensor product of linear differential systems. Cf. also \cite[Def. 2.20]{zbMATH01860838}.

Given a $G$-operator $L$, we may regard its solutions (possibly multivalued) as functions on open subsets\footnote{Here and elsewhere, we write $\PP^1(\C)$ for the complex points of $\PP^1$ equipped with the analytic topology, to distinguish it from $\PP^1$ with the Zariski topology.} of $\PP^1(\C)$. 
\begin{defn}
\label{sigmadef}
We write $\Sigma = \Sigma(L) \subset \mathbb{P}^1$ for the set where $L$ does not admit a basis of holomorphic solutions. This is also known as the set of \emph{non-apparent singularities} of $L$.
\end{defn}

Finally, we record here two pieces of notation that will enter in the next section.
\begin{defn}
\label{indexdef}
Let $L$ be a $G$-operator on $\PP^1$. To $L$ we associate:
\begin{itemize}
    \item The \emph{nilpotency index} of $L$, $N(L)$: the size of the largest Jordan block in the monodromy matrix $M$ at $z = 0$ of a fundamental matrix of solutions. It also corresponds to one more than the maximum power of the $\log$ appearing in the entries of a fundamental solution matrix.
    \item The \emph{ramification index} of $L$, $e(L)$: the common denominator of the rational exponents appearing in a fundamental solution, i.e. smallest integer $e$ such that $[z^e]^*M$ is unipotent. 
\end{itemize}
\end{defn}

\subsection{Differential operators of geometric
origin are $G$-operators, after André}\label{thmGoperator}

Let $S = \mathbb{P}^1_{\Qbar} \setminus \Sigma$, where $\Sigma \subset \mathbb{P}^1(\Qbar)$ is a finite set. Let $f: X \to S$ be a smooth proper morphism of $\Qbar$-varieties. For each $j \geq 0$, the relative algebraic de~Rham cohomology $\mathcal{H} := R^{j} f_{*} \Omega^{\bullet}_{X/S}$ is a locally free $\mathcal{O}_S$-module equipped with the \emph{Gauss–Manin connection} $\nabla: \mathcal{H} \to \mathcal{H} \otimes_{\mathcal{O}_S} \Omega^1_{S}$. Given a global coordinate $z$ on $S$ and a nonzero section $\omega \in \mathcal{H}(S)$, the connection $\nabla$ yields a linear differential operator $L_\omega \in \Qbar[z][\partial]$ killing the periods associated to $\omega$. Any such operator of minimal order is called a \emph{Picard–Fuchs operator} of the family $f$.

\begin{defn}[{\cite[Ch. II, Def. 1.3]{zbMATH00041964}}]
\label{defn:geometric-origin}
A linear differential operator $L \in \Qbar[z][\partial]$ is said to be of \emph{geometric origin} if it is a product of factors of Picard–Fuchs operators. 
\end{defn}

Note that, as is explained in \cite[Ch. II]{zbMATH00041964} (cf. also the proof appearing in \cite[Ch. V, App.]{zbMATH00041964}), the notion of geometric here applies to differential operators coming from iterated extensions of Picard-Fuchs equations arising from smooth and proper families; in particular, it will apply in situations arising from families of algebraic varieties, not necessarily proper or smooth.

\begin{thm}[André {\cite[Ch. V, App.]{zbMATH00041964}}]
\label{thm:andre-geometric}
Let $L \in \overline{\mathbb{Q}}[z][\partial]$ be a differential operator of geometric origin. Then $L$ is a $G$-operator. 
\end{thm}

The converse question—whether \emph{every} $G$-operator in $\overline{\mathbb{Q}}[z][\partial]$ is of geometric origin—is an open conjecture due to Bombieri and Dwork.

\subsection{Bombieri and André height bound}

In this section we review a strengthening of a result of Bombieri, cf. \cite[\S11, \S12]{zbMATH03721021}, which can be found in \cite[VII, Sec. 4.3] {zbMATH00041964}.

Let $G_1,\dots,G_{\mu} \in K[[z]]$ be $G$-functions with $K \subset \Qbar$ a number field. We fix $v$ a place of $K$; let $K_{v}$ be the completion of $K$ and let $\xi \in \overline{K_{v}}$ be a point whose $v$-adic valuation lies strictly inside the radius of convergence of all the $G_{i}$. 

\begin{defn}(Bombieri, Andr\'e) 
We say that there is a \emph{strongly non-trivial relation of degree $\delta$} among the values $G_i(\xi)$ if there exists a polynomial $R \in K[T_1,\dots,T_{\mu}]$ of degree $\delta$ such that
\[
R\bigl(G_1(\xi),\dots,G_{\mu}(\xi)\bigr)=0,
\]
holds $v$-adically, and such that this relation is not induced by a relation that holds identically among the functions $G_i$ themselves.
\end{defn}

To be more precise, let $B_{\mathbf{G}} \subset \mathbb{P}^1 \times \mathbb{A}^{\mu}$ be the Zariski closure of the graph of the tuple $\mathbf{G} = (G_{1}, \hdots, G_{\mu})$, where we regard each of the $G_{i}$ as functions on a common open subset of $\mathbb{P}^1(\C)$. 
\begin{defn}
\label{nontrivialreldef}
A \emph{non-trivial} relation $R$ of degree $\delta$ is a hypersurface $H_{R} \subset \mathbb{A}^{\mu}$ of degree $\delta$ such that the specialized point $(G_1(\xi), \hdots, G_{\mu}(\xi))$ lies in $H_R$ (regarded inside $\{ \xi \} \times \mathbb{A}^{\mu})$, and such that the intersection $B_{\mathbf{G},\xi} \cap H_{R}$ has dimension smaller than $\dim B_{G} - 1$. 
\end{defn}
\noindent The above notion is called \emph{strongly non-trivial} in \cite{zbMATH00041964}.

\medskip

\begin{thm}[Bombieri, André]\label{thm:bombieri}
Let $G_1,\dots,G_{\mu}$ be $\mu$ $G$-functions over a number field $K \subset \Qbar$, and let $h(-) : \Qbar \to \mathbb{R}_{\geq 0}$ denote the standard logarithmic Weil height. Let $\mu'$ denote the transcendence degree of the field over $K(z)$ generated by the $G_i$ and all their derivatives. Then there exist two constants $c_1,c_2>0$, depending only on the $G_i$, with the following property.

Let $\delta$ be a positive integer. If for some $\xi \in K$ we have
\begin{equation}
\label{bomb1}
    h(\xi) \ge c_1\, \delta^{\mu'}(\log \delta +1)
\end{equation}
and
\begin{equation}
\label{bomb2}
 |\xi|_v \le \exp\!\left(-c_2\, \delta^{\,1-\frac{1}{\mu'+1}}\, h(\xi)^{\,\frac{1}{\mu'+1}} \log h(\xi)^\frac{1}{\mu'+1}\right),   
\end{equation}
then there is no non-trivial relation of degree $\delta$ among the numbers $G_i(\xi)$ in $K_v$.
\end{thm}

The above statement is proved in the course of arguments taking place in \cite[VII]{zbMATH00041964}. Note that \cite[VII, \S4.1]{zbMATH00041964} has the running hypothesis that the entries of the vector $\mathbf{G}$ are algebraically independent over $K(z)$, but this hypothesis is removed in \cite[VII, \S4.2]{zbMATH00041964}, and the above statement appears in \cite[VII, \S4.3]{zbMATH00041964}.

\begin{rem}
In \cite[VII, Sec. 4.3]{zbMATH00041964} only \emph{homogeneous} relations are considered, and the $G$-functions are numbered starting from $0$ up to $\mu - 1$. The statement we give is obtained by replacing $\mu-1$ with $\mu$ and setting $G_{0} = 1$, which allows one to remove the term ``homogeneous''. 
\end{rem}

\section{Main result on $G$-functions and $G$-operators}\label{Gresult}

\subsection{Statement of the main theorem}
\label{Gresultstatementsec}

Let $K \subset \Qbar \subset \C$ be a number field, which we fix throughout. We now label the coordinate $z$ from the previous section as $s$, and set $x = 1/s$. Consider a $G$-operator $L\in K[s][\partial]$, in the sense of \autoref{Gopdef}. Set $n$ for the rank of $L$. Let $\Sigma = \Sigma(L)$ be as in \autoref{sigmadef}. Fix a basis of solutions $b_{1}, \hdots, b_{n}$ of $L$ as well as a basis $m_{1}, \hdots, m_{n}$ for the associated differential module $\mathcal{D} / \mathcal{D} L$. Let 
\begin{equation}\label{basisorigi}
    \mathbf{P} = [b_j(m_i)]_{1 \leq i, j \leq n}
\end{equation}
be the associated fundamental matrix. We define $B_\mathbf{P} \subset \mathbb{P}^1 \times \GL_{n} \subset \mathbb{P}^1 \times \mathbb{A}^{n^2}$ as the $K$-Zariski closure of the graph of the fundamental solution $\mathbf{P}$. For any evaluation point $s \in \PP^1(\C)$, let $B_{\mathbf{P}, s} \subset \mathbb{A}^{n^2}$ denote the algebraic fibre of $B_\mathbf{P}$ over $s$.

We want to consider relations $H_{R}$ on $\mathbf{P}(s)$ that do not vanish along any irreducible component of the germ of $B_{\mathbf{P},s} (\C)$ at $\mathbf{P}(s)$. This means that, for any open neighbourhood $\mathcal{N} \subset B_{\mathbf{P},s}(\mathbb{C})$ of $\mathbf{P}(s)$ and any irreducible component $C$ of $B_{\mathbf{P},s}$ passing through $\mathbf{P}(s)$, we have $C(\C) \cap \mathcal{N} \nsubseteq H_R(\C) \cap \mathcal{N}$. Equivalently, in algebro-geometric terms, we have:  

\begin{defn}\label{def:nontrivial}
Given $s \in \PP^1(K)$ a \emph{relation} at $s$ is a $K$-algebraic hypersurface $H_{s} = H_{R}$ in $\{ s \} \times \mathbb{A}^{n^2}$ defined by the polynomial $R$. It is said to be \emph{non-trivial} for $\mathbf{P}$ if $R$ does not vanish on any $K$-irreducible component of $B_{\mathbf{P},s}$ passing through $\mathbf{P}(s)$.

The \emph{degree} of the relation is the degree of $R$. 
\end{defn}

\noindent Note that the above is a version of \autoref{nontrivialreldef} except for $B_{\mathbf{P}}$ instead of $B_{\mathbf{G}}$.

Let $L$ be a $G$-operator on $\mathbb{P}^1$ of rank $n$ as above, let $x$ be the coordinate of $\mathbb{P}^1$ at $0$, and $s = 1/x$ so that $\PP^1 = \operatorname{Proj}\, K[x, s]$. Let $\mathbf{P}(s)$ be a fundamental matrix of analytic solutions of $L$ in the parameter $s$, which we regard as single-valued analytic functions on a punctured disk around $\infty = [1 : 0]$, making a branch cut if necessary. 

\begin{thm}\label{thmwithproof}
Let $\delta \ge 1$ be a natural number and consider the set
\begin{equation}
\mathcal{R}_{\mathbf{P},\delta} := \left\{ x \in \mathbb{P}^1(K) \mid\; \begin{array}{c} \mathbf{P}(x) \text{ satisfies no relations of degree}\leq \delta \\ \text{which are non-trivial for $\mathbf{P}$.} \end{array} \right\}
\end{equation}
Then:
\begin{enumerate}
    \item[(1)] $\mathcal{R}_{\mathbf{P},\delta}$ has $\infty$ as an accumulation point. 
    \item[(2)] There exist an integer $D \leq n^2 (|\Sigma(L)|+1)^2 + 1$, and a constant $\kappa$, such that for any integer $t > \kappa$, at least one of the following points
    \begin{displaymath}
        [t:1], [t+D:1],\dots, [t+(n^2+1)D:1] \in \PP^1(K)
    \end{displaymath}
     lies in $\mathcal{R}_{\mathbf{P},\delta}$. 
\item[(3)]
More generally, for each positive integer $\nu$, there exist
an integer
\[
 1\leq D=D(\nu)\leq
 (n^2+\nu)(|\Sigma(L)|+1)^2+1
\]
and a constant $\kappa=\kappa(\nu)$ such that, for every
integer $t>\kappa$, at least $\nu$ of the points
\[
 [t:1],[t+D:1],\ldots,[t+(n^2+\nu)D:1]
\]
lie in $\mathcal{R}_{\mathbf P,\delta}$.

\end{enumerate}
\end{thm}

The rest of \S\ref{Gresult} is devoted to the proof of \Cref{thmwithproof}.

\subsubsection{A monodromy lemma}\label{sec:monolemma}

Expanding on our sketch from \S\ref{strsec}, the first step is to find a rational curve $\overline{C} \subset (\PP^1)^{N}$ where $N = n^2 + 2$ (the choice of this particular $N$ will be explained in \Cref{globalrelcodimdropsec}), such that denoting $C = \overline{C} \cap (\PP^1 - \Sigma(L)\cup\{0\})^N$ we have that:
\begin{itemize}
    \item[(a)] the natural map $\pi_1(C) \to \pi_1(\PP^1 - (\Sigma(L)\cup\{0\}))^N$ is surjective;
    \item[(b)] the curve $\overline{C}$ is defined over $K$; and
    \item[(c)] $\overline{C}$ passes through the point $(\infty, \cdots, \infty)$ and is tangent to the diagonal at this point.
\end{itemize}
The addition of the point $\{ 0 \}$ to $\Sigma(L)$ will only be necessary in the situation where $\infty \in \Sigma(L)$; since in \S\ref{strsec} we assumed this was not the case, this issue did not arise there. The reader interested only in (1) of \Cref{thmwithproof} can just apply the Bertini theorem (and weak Lefschetz) to find such a $\overline{C}$. Moreover, if such a reader is only interested in the subsequent applications to analytic density, then they may also assume that $s = 0$ is a regular point of the differential operator $L$ (i.e., that $\infty \notin \Sigma$) and consequently ignore both the addition of $\{ 0 \}$ and the appearance of logs and rational exponents in the arguments of the coming subsections.

For the more precise estimate appearing in (2) and (3) we need the following explicit lemma to choose $\overline{C}$.
\begin{lem}\label{lemmashift}
Let $m > 0$ be an integer, and let $\Sigma' \subset \PP^1(\C)$ be a finite set. Then there is an integer $D \leq m^2 |\Sigma'|^2 + 1$ such that the curve $\overline{C} = \overline{C}_D \subset (\PP^1 )^{m^2+1}$ defined as the image of
\begin{displaymath}
     x \mapsto (x, x+D, \dots, x+m^2D)
\end{displaymath}
has full monodromy, i.e., the map $\pi_1(C) \to \pi_1(\PP^1 - \Sigma')^{m^2 + 1}$ is surjective, where $C = \overline{C} \cap (\PP^1 - \Sigma')^{m^2 + 1}$.
\end{lem}

\begin{proof}
We select an integer $D$ such that the translated sets $(\Sigma' \setminus \{ \infty \}) - jD$ are pairwise disjoint for $j = 0, 1, \dots, m^2$. This disjointness is satisfied provided that $jD \neq \sigma_1 - \sigma_2$ for any $\sigma_1, \sigma_2 \in (\Sigma' \setminus \{ \infty \})$. As there are at most $|\Sigma'|^2$ such differences and $m^2$ non-zero multipliers, there are at most $m^2|\Sigma'|^2$ forbidden values for $D$, showing that such a $D$ can be chosen with $D \le m^2|\Sigma'|^2 + 1$. 

Now $\overline{C} = \overline{C}_{D}$ passes through the point $(\infty, \hdots, \infty)$, so the required surjectivity on fundamental groups reduces to showing the surjectivity of the map 
\[ \pi_1(C) \to \pi_1(\mathbb{A}^1 - (\Sigma' \setminus \{ \infty \}))^{m^2 + 1} . \]
The fundamental group of $\pi_1(C)$ is generated by loops around $(m^2 + 1) |\Sigma' \setminus \{ \infty \}|$ punctures, and the fundamental group of each $\mathbb{A}^1 - (\Sigma' \setminus \{ \infty \})$ is generated by the corresponding loops around points of $\Sigma'$. By construction the map on fundamental groups sends generators to generators, so the result follows.
\end{proof}

For the rest of the section we focus on the proof of \autoref{thmwithproof}(2), which means we will work with the above $\overline{C}_D$ curve with $\Sigma' = \Sigma \cup \{ 0 \}$, rather than an unspecified curve $\overline{C}$ that satisfies conditions (a), (b) and (c) above. The proof of \autoref{thmwithproof}(1) with an arbitrary curve $\overline{C}$ satisfying (a), (b) and (c) as above proceeds similarly with minor modifications. (The tangency condition in (c) is necessary to handle the logarithmic factors appearing in the solutions of $L$, and to carry out a calculation analogous to \eqref{eqlog}.) We prefer the more explicit approach as we are interested in more precise density statements; note that \autoref{thmwithproof}(2) already implies \autoref{thmwithproof}(1). 
\subsection{Preliminary constructions}

The points $x = t$ in the second part of \autoref{thmwithproof} correspond to the parameter values $s = 1/t$. To study the values of $\mathbf{P}(1/t)$, we construct various functions centered at $\infty$ and express them in terms of the uniformizing parameter $s$. A priori these functions may be multi-valued and undefined at $\infty$ due to the presence of logarithmic factors $\log(s)$ and roots of $s$, so we instead work on a punctured disk around $\infty$ with a branch cut, where we assume the branch cut does not contain any $K$-points. This branch cut is chosen compatibly with the one that we used to evaluate $\mathbf{P}$. We write $A \subset \PP^1(\mathbb{C})$ for this punctured disc. 

\subsubsection{An auxiliary $G$-operator}
\label{auxGopsec}

Put $b=N(L)\geq1$. Choose a fundamental matrix
$\mathbf{G}_{\log}$ for $L$ over $K$. We may write
\begin{equation}
 (\mathbf{G}_{\log})_{ij}(s)
 =\sum_{r=0}^{b-1}G_{ij}^{(r)}(s)(\log s)^r,
 \qquad
 G_{ij}^{(r)}(s)=s^{\alpha_j}f_{ijr}(s),
\end{equation}
where $\alpha_j \in \mathbb Q$ and the $f_{ijr}$ are $G$-functions with coefficients in $K$. There is a constant matrix
$Q \in \mathrm{GL}_n(\mathbb C)$ such that
$\mathbf P=\mathbf{G}_{\log}Q$ on $A$.

Let $\mathbf{N}_b$ be the standard nilpotent Jordan block of size $b+1$, and set
\begin{equation}\label{defJ}
 \mathbf{U}_b(s)=\exp((\log s)\mathbf{N}_b),
 \qquad
 \mathbf{J}(s)=\mathbf{G}_{\log}(s)\otimes \mathbf{U}_b(s).
\end{equation}
We write $L_{\operatorname{log}}$ for the natural differential operator whose solutions are spanned by 
\[ 1, \log(s), \log(s)^2, \hdots, \log(s)^{b} . \]
Then $\mathbf{U}_{b}(s)$ is a fundamental matrix for $L_{\operatorname{log}}$. Setting $L_G = L \otimes L_{\operatorname{log}}$, we may observe that $\mathbf{J}$ is a fundamental matrix for $L_{G}$ of
size $n(b+1)$. 

The solution vector $\mathbf{J}(s)$ is defined on
$\mathbb P^1\setminus\Sigma'$, where
$\Sigma'=\Sigma(L)\cup\{0,\infty\}$.
Write $\mathbf{G}$ for the vector of log-free coefficients
$G_{ij}^{(r)}$; its length is at most $n^2b$.

\begin{lem}\label{lem4.4}
Every entry of $\mathbf{G}$ is a $K$-linear combination of entries
of $\mathbf{J}$. Moreover, some nonzero entry $T$ of $\mathbf{G}$ is a
$K$-linear combination of entries of $\mathbf{G}_{\log}$, and
$T \log s$ is a $K$-linear combination of entries of $\mathbf{J}$.
\end{lem}

\begin{proof}
In what follows we treat the expression $\ell=\log s$ formally. Then differentiation with respect
to $\ell$ commutes with differentiation in $s$, which acts
as $\partial_s+s^{-1}\partial_\ell$ on
$K((s^{1/e}))[\ell]$. Hence $\partial_\ell$ preserves the
$K$-vector space of solutions to the differential operator $L_{G}$, and
$\partial_\ell^a \mathbf{G}_{\log}=\mathbf{G}_{\log} C_a$ for constant matrices $C_a$ over $K$.

Setting $\mathbf{G}^{(r)}= [G_{ij}^{(r)}]$, one computes that
\begin{equation}
 \mathbf{G}^{(r)}
 =\frac{1}{r!}
 \sum_{k=0}^{b-1-r}
 \frac{(-\ell)^k}{k!}\,
 \partial_\ell^{\,r+k} \mathbf{G}_{\log}.
\end{equation}
The entries on the right are $K$-linear combinations of
entries of $\mathbf{J}$, proving the first assertion.

The highest non-zero logarithmic coefficient vector of any
column of $\mathbf{G}_{\log}$ is a solution of $L$. Choose a nonzero coordinate $T$ of this vector. Then $T$ is a $K$-linear combination of entries of
$\mathbf{G}_{\log}$. Since $b \geq 1$, the entries of $\mathbf{U}_b$ include
$\log s$, which proves the assertion about $T\log s$.
\end{proof}

\subsubsection{Bundle constructions} 
\label{bundconstrsec}

\paragraph{Zariski closures of solutions:} We recall the following general setup. Suppose that $M$ is a differential operator of rank $n$, and that $\mathbf{M} : A \to \GL_{n}(\mathbb{C})$ is a fundamental matrix of solutions of $M$. Then it is a general fact (cf. \cite[Ch 1, Thm. 1.28]{zbMATH01860838}) that the (complex) Zariski closure of the graph of $\mathbf{M}$ in $(\mathbb{P}^1 - \Sigma(M)) \times \GL_{n}$ is a torsor for the differential Galois group of $M$, which naturally acts on the right. When $M$ has regular singularities, this differential Galois group is also the algebraic monodromy group of $M$. We denote this group by $H = H_{M}$, and write $H^{\circ}$ for its identity component. 

For the matrix $\mathbf{J}$ constructed in \eqref{defJ}, let $B_\mathbf{J}$ be the
complex Zariski closure of its graph in
$\mathbb P^1\times\mathrm{GL}_{n(b+1)}$.
Over the locus $\mathbb{P}^1 - \Sigma(L_{G})$, this is a torsor
for the algebraic monodromy group of $L_G$.
Similarly, let $B_{\mathbf{G}_{\log}}$ be the complex Zariski
closure of the graph of $\mathbf{G}_{\log}$ in
$\mathbb P^1\times\mathrm{GL}_n$; over $\mathbb{P}^1 - \Sigma(L)$, it is a
torsor for the algebraic monodromy group of $L$.

\begin{lem}\label{lemdefoverK} 
The complex subvariety $B_{\mathbf{J}} \subset \mathbb{P}^1 \times \GL_{(N(L) + 1) n}$ (resp. $B_{\mathbf{G}_{\operatorname{log}}}$) is defined over $K$. In particular, it is equal to the $K$-Zariski closure of the graph of $\mathbf{J}$ (resp. $\mathbf{G}_{\operatorname{log}}$) on $A$.
\end{lem}
\begin{proof}
We prove the claim for $B_{\mathbf{J}}$, as the argument for $B_{\mathbf{G}_{\operatorname{log}}}$ is analogous. Consider a complex analytic function $F(s)$ which is a polynomial in $\log(s)$ with coefficients in $K((s^{1/e}))$. Given a complex polynomial $R \in \mathbb{C}[x_{1}, \hdots, x_{\ell}]$ and objects $F_{1}, \hdots, F_{\ell}$ of this type, the relation $R(F_{1}, \hdots, F_{\ell}) = 0$ holds if and only if after expanding the result as a polynomial in $\log(s)$, each coefficient vanishes in the ring $\mathbb{C}((s^{1/e}))$. This condition is stable under replacing $R$ with $R^{\sigma}$ for any automorphisms $\sigma \in \textrm{Aut}(\mathbb{C}/K)$. The ideal defining the intersection of $B_{\mathbf{J}}$ with $U \times \GL_{(N(L)+1)n}$ is spanned by polynomials of this type in the natural coordinates on the product. This ideal is therefore defined over $K$, hence so is the associated variety $B_{\mathbf{J}}$.
\end{proof}

We set $B_Q:=\overline{\{Q\}}^{\,K\text{-Zar}}\subset\mathrm{GL}_n$ and consider the algebraic map over $\mathbb{P}^1$
\begin{displaymath}
    v : B_{\mathbf{G}_{\operatorname{log}}} \times B_{Q} \subset \mathbb{P}^1 \times \GL_{n} \times \GL_{n} \to \mathbb{P}^1 \times \GL_{n}, \hspace{2em} (s, A, B) \mapsto (s, A \cdot B) .
\end{displaymath}
For each $s \in A\subset \PP^1(\mathbb{C})$ we have $v(s, \mathbf{G}_{\operatorname{log}}(s), Q) = (s, \mathbf{P}(s))$. Since  $B_{\mathbf{G}_{\operatorname{log}}}$ is defined over $K$ (by the above lemma), so is $v$. As the image $I = v(B_{\mathbf{G}_{\operatorname{log}}} \times B_{Q})$ contains the graph of $\mathbf{P}$ it follows that the $K$-Zariski closure $\overline{I}^{\textrm{Zar}}$ contains $B_{\mathbf{P}}$. 

\begin{lem}\label{lem46}
We have $\overline{I}^{K-\operatorname{Zar}} = B_{\mathbf{P}}$.
\end{lem}

\begin{proof}
It suffices to show the other inclusion. Since $v^{-1}(B_{\mathbf{P}})$ contains all the points in the set $W = \{(s, \mathbf{G}_{\operatorname{log}}(s), Q) : s \in A \}$, it suffices to show that the $K$-Zariski closure of $W$ is $B_{\mathbf{G}_{\operatorname{log}}} \times B_{Q}$. Observe from \autoref{lemdefoverK} that this is the same as the $K$-Zariski closure of $B_{\mathbf{G}_{\operatorname{log}}} \times \{ Q \}$. But any $K$-algebraic function vanishing on $B_{\mathbf{G}_{\operatorname{log}}} \times \{ Q \}$ vanishes identically on the first factor, and therefore comes from a function vanishing on $B_{Q}$.
\end{proof}
From now on we consider $v$ as a $K$-algebraic map $B_{\mathbf{G}_{\operatorname{log}}} \times B_{Q} \to B_{\mathbf{P}}$. 

\begin{lem}
\label{newirredlem}
For $s\in \mathbb{P}^1(K) \setminus \Sigma(L)$ and every $K$-irreducible component $D$ of
$B_{\mathbf{G}_{\log},s}\times B_Q$, the $K$-Zariski closure of $v(D)$
is a $K$-irreducible component of $B_{\mathbf{P},s}$.
\end{lem}
\begin{proof}
Over $\mathbb C$, write the torsor $B_{\mathbf{G}_{\log},s}$ as $tH$
and the geometric components of $B_Q$ as $Z_i$. The closures
of the images of the geometric product components are all of the form $\overline{thH^\circ Z_i}^{\mathbb{C}-\textrm{Zar}}$. These closures are all irreducible and have the same dimension. Indeed, by left translation one can reduce to considering just the various $\overline{H^\circ Z_i}^{\mathbb{C}-\textrm{Zar}}$, and then using that the $Z_i$ are all Galois conjugate and $H^\circ$ is defined over $K$ we see that each such closure has the same dimension as $\overline{H^\circ Z_1}^{\mathbb{C}-\textrm{Zar}}$. The union of all these components is $B_{\mathbf P,s}$: to see this, locally analytically trivialize
$B_{\mathbf{G}_{\log}}$ to identify the image of $v$ with a product
of an open disk and $HB_Q$, then take its closure, using \Cref{lem46} to equate the complex algebraic and $K$-algebraic closures.
Thus the various $\overline{thH^\circ Z_i}^{\mathbb{C}-\textrm{Zar}}$ are precisely the geometric components of
$B_{\mathbf P,s}$. The result on $K$-irreducible components then follows by taking Galois orbits.
\end{proof}

\paragraph{Product Constructions and $G$-functions:} Apply \autoref{lemmashift} with $m=N=n^2+2$ and
$\Sigma'=\Sigma(L)\cup\{0,\infty\}$, and take $D$ and
$C=\overline C_D$ as in that statement. We define the $N = n^2 + 2$  shifted evaluation maps for $j = 0, \dots, n^2+1$:
\[ a_j(s) = \frac{s}{1 + jDs} . \]
Evaluating at the origin yields $a_j(0) = 0$ for all $j$. We set $\mathbf{P}_{j}$ to be the pullback of $\mathbf{P}$ by $a_j$. Similarly, we write $\mathbf{G}_{j}$ for the pullback of $\mathbf{G}$ from \S\ref{auxGopsec}, and $\mathbf{J}_j$ for the pullback of $\mathbf{J}$. By mimicking the above constructions, we obtain analogous bundles $B_{\mathbf{P}}(j)$ and $B_{\mathbf{J}}(j)$ with $B_{\mathbf{P}} = B_{\mathbf{P}}(0)$ and $B_{\mathbf{J}} = B_{\mathbf{J}}(0)$. We also write $u_{j} : B_{\mathbf{J}}(j) \times B_{Q} \to \mathbb{P}^1 \times \GL_{n}$ for the corresponding maps; once again if $I(j)$ is the image of $u_{j}$ then $\overline{I(j)}^{\operatorname{Zar}}$ equals $B_{\mathbf{P}}(j)$ and $I(j)$ is a topological fibre bundle away from the singular locus of $a^{*}_{j} L$. Note that our fixed curve $\overline{C}$ coming from \autoref{lemmashift} is the image of the map $a : \PP \to \PP^{N}$ whose components are $a_{j}$. The lemmas \autoref{lemdefoverK}, \autoref{lem46} and \autoref{newirredlem} also apply to the bundles $B_{\mathbf{B}_{\operatorname{log}}}(j)$ and $B_{\mathbf{P}}(j)$ up to replacing $\Sigma(L_{G})$ with $\Sigma(a^{*}_{j} L_{G})$.

We define $\mathcal{B}_{\mathbf{J}}$ as the fibered product of the pullback bundles $B_{\mathbf{J}}(j)$ over the base curve $\mathbb{P}^1$:
\[ \mathcal{B}_{\mathbf{J}} = B_{\mathbf{J}}(0) \times_{\mathbb{P}^1} B_{\mathbf{J}}(1) \times_{\mathbb{P}^1} \dots \times_{\mathbb{P}^1} B_{\mathbf{J}}(n^2+1) \]
As we know from \autoref{lemdefoverK} each $B_{\mathbf{J}}(j)$ is defined over $K$, hence so is this fibre product. Moreover, the Zariski closure of the graph of the vector-valued function
\begin{equation}
\label{BGgrapheq}
s \mapsto \{ (\mathbf{J}_{j}(s)) \}^{n^2+1}_{j=0} 
\end{equation}
contains the analytic continuation of this function along all paths in the curve $C$, and $C$ was chosen such that the map $\pi_1(C) \to \pi_1(\mathbb{P}^1 - \Sigma')^{N}$ is surjective, so it follows that this Zariski closure is equal to $\mathcal{B}_{\mathbf{J}}$; note that, since $\mathcal{B}_{\mathbf{J}}$ is defined over $K$, there is no distinction between $K$-Zariski and complex Zariski closures in this case. Since each $B_{\mathbf{J}}(j)$ is a torsor for the algebraic monodromy group of $a_{j}^{*} L_{G}$, the bundle $\mathcal{B}_{\mathbf{J}}$ is a torsor for the product of these groups. We also construct the bundle 
\begin{displaymath}
     \mathcal{B}_{\mathbf{P}} = B_{\mathbf{P}}(0) \times_{\mathbb{P}^1} B_{\mathbf{P}}(1) \times_{\mathbb{P}^1} \dots \times_{\mathbb{P}^1} B_{\mathbf{P}}(n^2+1) . 
\end{displaymath}
We obtain a map $\mathcal{u} : \mathcal{B}_{\mathbf{J}} \times B_{Q} \to \mathcal{B}_{\mathbf{P}}$. 
Choose $T$ as in \Cref{lem4.4} and put $e=e(L)$.
For $1\leq j<N$, define
\begin{equation}\label{eqlog}
 E_j(s)
 =(T(s)T(a_j(s)))^e
   \bigl(\log s-\log(a_j(s))\bigr)
 =(T(s)T(a_j(s)))^e\log(1+jDs).
\end{equation}
Each coefficient $F$ occurring in a vector $G_j$ has a
single fractional exponent, so $F^e$ is unramified.
Choose a fixed integer $c\geq0$ clearing all poles at zero.
Let $\mathbf H$ consist of the functions $s^cF^e$ and
$s^cE_j$. These are $G$-functions over $K$.

By \Cref{lem4.4}, each $F$, and each of
$T(a_j(s))$ and $T(a_j(s))\log(a_j(s))$, is a
$K$-linear combination of entries of $\mathbf{J}_j$.
Consequently the entries of $\mathbf H$ are regular
algebraic functions on $\mathcal B_\mathbf{J}$ over $U$,
after shrinking $U$ if necessary.
Let $\mathcal B_\mathbf{H}$ be the Zariski closure of their graph.
We obtain a dominant $K$-algebraic map
$\pi:\mathcal B_\mathbf{J} \to \mathcal B_\mathbf{H}$.

The normalized powers recover all the coefficients $F$
algebraically. They therefore recover $T(s)$ and
$T(a_j(s))$ algebraically as well. Dividing the recovered
$E_j$ by the generically nonzero factor
$(T(s)T(a_j(s)))^e$ gives $\log(1+jDs)$.
After adjoining $\log s$, we thus recover
$\log(a_j(s))$ and every entry of every $\mathbf{J}_j$ algebraically.
It follows that
\begin{equation}\label{4.8}
 \dim\mathcal B_\mathbf{J}-1\leq\dim\mathcal B_\mathbf{H}.
\end{equation}

\subsection{Proof of \Cref{thmwithproof}}

We let $t$ be an integer, and $s = 1/t$. To apply \autoref{thm:bombieri} to polynomial relations on our product system, we argue as follows.

\subsubsection{Global relations and codimension drop} 
\label{globalrelcodimdropsec}

We start with the following abstract lemma; its application to our situation is explained below.

\begin{lem}\label{deeplemma}
Let $X_1,\ldots,X_N$ be smooth equidimensional affine varieties
over $K$, put $X=\prod_jX_j$, and let $q_{0} \in\mathbb A^r(\mathbb{C})$ be a point. Suppose $R_j$ for $1 \leq j \leq N$ is a $K$-algebraic function on $X_j\times\mathbb A^r$, and let $E=(x_1,\ldots,x_N,q_0)$ be a common complex zero of the $R_{j}$. Assume that $R_j(-,q_0)$ is nonzero in the
local ring of $X_j$ at $x_j$ for every $j$.
Then a $K$-irreducible component $Z$ of the common zero locus through $E$ of the $R_{j}$ has codimension $N$. The closure $Y$ of its projection
to $X$ has codimension at least $N-r$ and contains
$(x_1,\ldots,x_N)$. For fixed affine embeddings of the $X_{i}$, the locus $Y$ is defined by polynomials whose degrees are polynomial in the degrees of the $R_j$.
If the $X_{i}$ and the embeddings vary in an algebraic family, this degree bound can be chosen uniformly in the family.
\end{lem}

\begin{proof}
We may work in the formal power series ring at $E$, which admits a presentation
\[ \mathbb{C}[[\{v^{i}_{1}\}^{r_{1}}_{i=1}, \hdots, \{v^{i}_{N}\}^{r_{N}}_{i=1}, y_{1}, \hdots, y_{\mu}]]/(h_{1}, \hdots, h_{\ell}) \]
where the $\{ v^{i}_{j} \}^{r_{j}}_{i=1}$ are formal coordinate systems at $x_{j}$ for the variety $X_{j}$, and the generators of the ideal $(h_{1}, \hdots, h_{\ell})$ defining the quotient only relates the $y_{i}$. By Weierstrass preparation and the fact that $R_{j}(-,q_{0})$ is non-zero at $x_{j}$ we can assume that, up to a change of formal coordinates, $R_{j}$ is monic in $v^{1}_{j}$. Taking the quotient by all the $R_{j}$, we end up with a ring which is finite free over
\[ \mathbb{C}[[\{v^{i}_{1}\}^{r_{1}}_{i=2}, \hdots, \{v^{i}_{N}\}^{r_{N}}_{i=2}, y_{1}, \hdots, y_{\mu}]]/(h_{1}, \hdots, h_{\ell}) . \]
It follows that every local component through $E$ of the locus defined by the $R_{j}$ has codimension $N$; this is likewise true for any $K$-component $Z$ passing through $E$ since it is a union of geometric components.

If we project $Z$ to $X$ the codimension drops by at most $r$. The remaining facts about degrees are immediate from the behaviour of degrees of varieties under intersection and projection.
\end{proof}

Now we explain how we will apply \autoref{deeplemma} to our situation. Put $N=n^2+2$, and suppose that, for each $0\leq j<N$,
$\mathbf P(a_j(s))$ satisfies a non-trivial $K$-relation
$R_j$ of degree at most $\delta$. Set $X_j=B_\mathbf{J}(j)_s$, and let
$p_j:X_j\to B_{\mathbf{G}_{\log},a_j(s)}$ be the natural projection.
For an arbitrary matrix $q\in\mathbb A^{n^2}$, set $\widetilde R_j(A,q)=R_j(p_j(A)q)$.

By \autoref{newirredlem}, the component $D$ of $B_{\mathbf{G}_{\operatorname{log}}, s} \times B_{Q}$ through
$(\mathbf{G}_{\log}(a_j(s)),Q)$ maps densely to a component $C = \overline{v(D)}^{\operatorname{Zar}}$ of
$B_{\mathbf{P},a_j(s)}$ through $\mathbf{P}(a_j(s))$.
The assumption that $R_{j}$ be non-trivial means that it is non-zero when restricted to $C$, and remains non-zero when pulled back to $D$. Furthermore, since $R_{j}$ is defined over $K$, the relation $R_{j}$ is also non-zero on the geometric component of $D$ passing through $(\mathbf{G}_{\log}(a_j(s)),Q)$, and also after pullback to $X_j$ along $p_{j}$.

Since $Q$ is $K$-generic in $B_Q$, it is generic over
$\overline{\mathbb Q}$ in its geometric component. The components of $X_j$ are defined over
$\overline{\mathbb Q}$. Consequently, $\widetilde R_j(-,Q)$ is nonzero on the germ at $\mathbf{J}_j(s)$: otherwise the relation would vanish on the entire component of $B_{\mathbf{J}_{j},s} \times B_{Q}$ passing through $(\mathbf{J}_{j}(s), Q)$.

Now apply \Cref{deeplemma} with $r = n^2$. Choose a $K$-irreducible component of the common zero locus through $(\mathbf{J}_0(s), \hdots, \mathbf{J}_{N-1}(s), Q)$, and let $Z_s$ be the closure of its projection to $\prod_j X_j=\mathcal B_{\mathbf{J},s}$. Then $Z_s$ contains $(\mathbf{J}_0(s), \hdots, \mathbf{J}_{N-1}(s))$, has codimension at least $N-n^2=2$, and has degree uniformly bounded in terms of $\delta$.

\subsubsection{Height bound for $t$ via Bombieri and Andr\'e}\label{heightboundsec}

Suppose that all $n^2+2$ values $\mathbf P(a_j(s))$, $0\leq j\leq n^2+1$, satisfy non-trivial $K$-relations of degree at most $\delta$, where $s=1/t$. By the preceding argument, the point $(\mathbf{J}_0(s), \hdots, \mathbf{J}_{N-1}(s))$ lies in a closed $K$-subvariety $Z_s\subset \mathcal{B}_{\mathbf J,s}$ of codimension at least two, with degree bounded independently of $s$.

Let $T_s\subset \mathcal{B}_{\mathbf H,s}$ be the closure of the image
of $Z_s$. It contains $\mathbf H(s)$ and has uniformly bounded degree. Since $\mathcal{B}_{\mathbf J,s}$ is equidimensional,
and every component of $\mathcal{B}_{\mathbf H,s}$ has dimension
$\dim \mathcal{B}_{\mathbf H}-1$, equation \eqref{4.8} gives
\[
 \dim T_s
 \leq \dim \mathcal{B}_{\mathbf J,s}-2
 < \dim \mathcal{B}_{\mathbf H}-1.
\]
Thus $T_s$ contains no component of $\mathcal{B}_{\mathbf H,s}$.
A generic $K$-linear combination of bounded-degree generators
of the defining ideal $I(T_s)$ therefore vanishes at $\mathbf H(s)$ without vanishing identically on any component of $\mathcal{B}_{\mathbf H,s}$. We obtain a non-trivial relation in the sense of \Cref{nontrivialreldef}, of degree at most some fixed integer $\rho\geq1$, independent of $s$.

We now apply \Cref{thm:bombieri} to the entries of $\mathbf H$.
Let $\mu'$ be as in that theorem and put $h=h(1/t)=\log t$. For all sufficiently large $t$,
\[
 h\geq c_1\rho^{\mu'}(1+\log\rho),
\]
so \eqref{bomb1} holds. Moreover,
\[
 \frac{
 c_2\rho^{\mu'/(\mu'+1)}
 (h\log h)^{1/(\mu'+1)}
 }{h}
 =
 c_2\rho^{\mu'/(\mu'+1)}
 \left(\frac{\log h}{h^{\mu'}}\right)^{1/(\mu'+1)}
 \longrightarrow0.
\]
Consequently, for all sufficiently large $t$,
\[
 |1/t|_v=e^{-h}
 \leq
 \exp\!\left(
 -c_2\rho^{\mu'/(\mu'+1)}
 (h\log h)^{1/(\mu'+1)}
 \right),
\]
which is \eqref{bomb2}. By monotonicity, these two inequalities
also hold with any integer $1\leq d\leq\rho$ in place
of $\rho$. \Cref{thm:bombieri} therefore excludes every
non-trivial relation of degree at most $\rho$ at
$\mathbf H(1/t)$, contradicting the relation constructed above.

Choose $\kappa$ large enough for these inequalities, for
convergence of $\mathbf H$ at $1/t$, and so that $1/t$ is not a singular point for any of the differential operators used in the preceding constructions. Since $
 a_j(1/t)=\frac{1}{t+jD}$,
for every integer $t>\kappa$ at least one of the points
\[
 [t:1],[t+D:1],\ldots,[t+(n^2+1)D:1]
\]
lies in $R_{\mathbf P,\delta}$. This proves (2), and
hence also (1).

\subsubsection{Proof of \Cref{thmwithproof}(3)}

Fix $\nu \geq 1$ and consider the $r=n^2+\nu+1$ shifted evaluation maps $a_j(s)=s/(1+jDs)$, $0\leq j<r$, where $D$ is as in \Cref{lemmashift}. Carry out the preceding proof to construct the analogous differential system and the vector $\mathbf H$ as above. For any subset $J\subset\{0,\ldots,r-1\}$ of size $n^2+2$, suppose that the evaluations $\mathbb{P}_{j}(s)$ for $j \in J$ all satisfy a non-trivial relation of degree $\leq \delta$. Apply the same argument as before using these $n^2+2$ relations, ignoring the other factors. Then the application of \Cref{deeplemma} still gives a subvariety $Z_{s} \subset \mathcal{B}_{\mathbf{J},s}$ of codimension at least two after eliminating $Q$, and the passage to $\mathcal{B}_{\mathbf{H}}$ loses at most one dimension. The previous height argument excludes this possibility for $t>\kappa_J$. Take $\kappa$ to be the maximum over the $\kappa_{J}$ for the finitely many subsets $J$. Then for $t > \kappa$ at most $n^2+1$ of the $r$ matrices $\mathbf{P}_{j}(s)$ satisfy a non-trivial relation of degree $\leq \delta$, so at least $\nu$ of these matrices satisfy no such relation. This proves (3).

\section{Hodge-theoretic recollections}\label{sec5}
We recall some facts about variations of Hodge structures, as studied by Griffiths \cite{Griffiths}, Deligne \cite{MR0498551}, Steenbrink and Zucker \cite{zbMATH04017061}, and many others.

\subsection{Variations of (mixed) $\mathbb{Z}$-Hodge structures and period maps}
\label{vhsandpermapsec}

Let $S$ be a smooth (irreducible) complex variety and denote its (complex) analytification by $S^{\an}=S(\C)$. A \emph{graded-polarizable variation of (mixed) $\Z$-Hodge structures} ($\Z$VHS) on $S$ is the data of $$\V:=(\V_\Q, \mathbb{W}_{\bullet}, \mathcal{V}, F^\bullet, \nabla),$$ where:
\begin{itemize}
    \item $\V_\Q$ is a finite rank locally free $\Q_{S^{\an}}$-local system on $S^{\an}$ admitting an integral structure $\V_{\Z} \subset \V_{\Q}$,
    \item $\mathbb{W}_{\bullet}$ is an increasing filtration of $\V_{\mathbb{Q}}$ by local subsystems (called the \emph{weight filtration}),
    \item $(\mathcal{V}, F^\bullet, \nabla)$ is a filtered regular algebraic flat connection on $S(\C)$ such that the analytification $(\mathcal{V}^{\an}, \nabla^{\an})$ is isomorphic to $\V_\mathbb{C} \otimes_{\Q_{S^{\an}}} \mathcal{O}_{S^{{\an}}}$ endowed
with the holomorphic flat connection defined by $\V_\Q$ (the
filtration $F^\bullet$ is called the \emph{Hodge filtration}),
\end{itemize}
such that the fibre of the tuple $(\V_\Z, \mathbb{W}_{\bullet}, \mathcal{V}, F^\bullet, \nabla)$ above each point $s \in S(\mathbb{C})$ is a graded-polarizable integral mixed Hodge structure (see e.g. \cite[\S 2]{MR1154159} for a definition). We will only consider $\Z$VHS that also satisfy two extra conditions. The first, \emph{Griffiths transversality}, was formulated by Griffiths \cite{Griffiths}, and applies to variations of pure Hodge structure obtained by passing to graded quotients with respect to $\mathbb{W}_{\bullet}$. The second, \emph{admissibility}, was formulated by Steenbrink and Zucker \cite{zbMATH04017061}. 

Let us fix an integral structure $\mathbb{V}_{\mathbb{Z}} \subset \mathbb{V}_{\mathbb{Q}}$. Associated to $\V$ and this integral structure there is a \emph{period map}, constructed as follows -- see also \cite[Sec. 4]{2021arXiv210708838B} for the pure case and \cite[Sec. 4]{2024arXiv240616628B} for the mixed one. Fix a graded-polarized lattice $(V, W_{\bullet}, Q)$ of the same type as the fibres of $\V$; for instance we could take $(V, W_{\bullet}, Q) = (\V_{\mathbb{Q}, s_{0}}, \mathbb{W}_{\bullet,s_{0}}, \mathcal{Q}_{s_{0}})$ for some point $s_{0} \in S(\mathbb{C})$. Let $D$ be the space of all graded-polarized Hodge structures on $(V, W_{\bullet}, Q)$ with the same Hodge numbers as those parameterized by $\V$. Let $\Gamma = \Aut(V, W_{\bullet}, Q)(\mathbb{Z})$ be the symmetry group of $(V, W_{\bullet}, Q)$. Then $\Gamma$ acts on $D$ and the quotient $\Gamma \backslash D$ is a complex orbifold with uniformization map $\pi : D \to \Gamma \backslash D$. We have a natural complex analytic map
\[ \varphi : S(\C) \to \Gamma \backslash D, \hspace{2em} s \mapsto (\V_{s}, \mathbb{W}_{\bullet,s}, F^{\bullet}_{s}, \mathcal{Q}_{s}) \]
where $F^{\bullet}_{s}$ denotes the Hodge flag at $s$. For later use we also note that the complex manifold $D$ embeds naturally into a flag variety $\ch{D}$ of Hodge flags on $V$ satisfying the first Hodge-Riemann bilinear relation on the associated graded $\textrm{gr}^{W} V$; see for instance \cite[pg.48]{zbMATH06031035} for a definition.

The map $\varphi$ admits a local description as follows. Let $(\mathcal{V}, F^{\bullet}, \nabla)$ be the algebraic filtered vector bundle with Gauss-Manin connection associated to $\V$. Let $B \subset S(\mathbb{C})$ be an analytic open subset on which $\V$ is trivial.

\begin{defn}
A \emph{filtration-compatible frame} $v_{1}, \hdots, v_{n}$ for $\mathcal{V}$ over an analytic or algebraic open subset $B$ of $S$ is a frame of $\restr{\mathcal{V}}{B}$ such that each filtered piece of $\restr{F^{\bullet}}{B}$ is spanned by a sequential subset of $\{ v_{1}, \hdots, v_{n} \}$; i.e., each $F^{i}$ is spanned by $v_{1}, \hdots, v_{n_{i}}$ for some $0 \leq n_{i} \leq n$. 
\end{defn}

\noindent Given a filtration-compatible frame $v_{1}, \hdots, v_{n}$ for $\restr{\mathcal{V}}{B}$, and a basis $b_{1}, \hdots, b_{n}$ for $\V(B)$, we can consider a varying change-of-basis matrix $\mathbf{P} : B \to \GL_{n}(\mathbb{C})$ such that $\mathbf{P}(s)$ gives the coordinates of $v_{1,s}, \hdots, v_{n,s}$ in the basis $b_{1}, \hdots, b_{n}$. By fixing a polarization-compatible identification $\iota$ between $\V(B)$ and the graded-polarized lattice $(V, W_{\bullet}, Q)$, we obtain a map 
\begin{equation}
\label{psidef}
\psi : B \to D \subset \{ \textrm{graded-polarized Hodge flags on }(V, W_{\bullet}, Q) \}, \hspace{2em} s \mapsto \iota(\mathbf{P}(s))
\end{equation}
where the matrix $\iota(\mathbf{P}(s))$ is interpreted as representing a flag on $V$. The maps $\psi$ are called \emph{local lifts} of the period map $\varphi$ and are independent of the choices modulo $\Gamma$.

\subsection{The Hodge locus}

Associated to every mixed $\Q$-Hodge structure $V$, there is a $\mathbb{Q}$-algebraic group $\MT(V)$ called its Mumford-Tate group. For a point $s \in S(\mathbb{C})$, we write $\mathbf{G}_s=\MT(\V_s)$ for the Mumford-Tate group of the Hodge structure above $s$. Note that $\mathbf{G}_s$ is constant (preserved by flat translation) away from a countable union of closed analytic subvarieties of $S$ (cf. \cite[Lem. 4]{MR1154159}). We call the associated abstract group the \emph{generic Mumford-Tate group}, and denote it $\mathbf{G}=\MT(\V)$, cf. also \cite[Def. 4.4]{2024arXiv240616628B}. For general properties of the Hodge locus we refer to \cite[\S4.2]{2024arXiv240616628B}.
\begin{defn}
Let $\V$ be a $\Z$VHS on $S$. The \emph{(tensorial) Hodge locus} of $\V$, $\HL(S,\V^\otimes) \subset S(\mathbb{C})$, is the locus of $s \in S(\mathbb{C})$ where $\mathbf{G}_s \neq \mathbf{G}$.
\end{defn}

\begin{defn}\label{hodgegen}
 A point $s\in S(\C)$ is called \emph{Hodge generic} if $\mathbf{G}_s = \mathbf{G}$. An irreducible subvariety $Y\subset S$ is called \emph{Hodge generic} if it contains a Hodge generic point.
\end{defn}

We now explain how one can reinterpret $\HL(S,\V^\otimes)$ using the period map $\varphi$. For each Mumford-Tate group $\mathbf{G}_{h}$ associated to some Hodge structure $h \in D$, the complex manifold $D$ contains a closed complex submanifold $D_{\mathbf{G}_{h}} \subset D$ consisting of all those Hodge structures with Mumford-Tate group contained in $\mathbf{G}_{h}$. If $\mathbf{G} = \mathbf{G}_{h}$ then one observes by construction that the image of $\varphi$ lies in the image of $D_{\mathbf{G}}$ in $\Gamma \backslash D$.

It is immediate from the definitions that $\HL(S, \V^{\otimes})$ is equal to the union
\[ \bigcup_{\substack{h \in D_{\mathbf{G}} \\ \mathbf{G}_h \subsetneq \mathbf{G}}} \varphi^{-1}(\pi(D_{\mathbf{G}_h})) . \]
For an equivalent description of the Hodge locus in terms of Mumford-Tate subgroups, rather than period domains, we refer to \cite[Sec. 1]{2021arXiv210708838B} and \cite[Sec. 4]{2024arXiv240616628B}.

\subsection{Preliminary Hodge-theoretic lemmas}\label{lemmas}
We collect some lemmas that we will use to prove the existence of algebraic Hodge generic points. Let $\V$ be a (mixed) $\Z$VHS as in \Cref{vhsandpermapsec} that satisfies the Griffiths transversality and the admissibility condition. These conditions are needed to invoke the \emph{theorem of the fixed part} of Deligne \cite[Sec. 4.1]{MR0498551}, Griffiths-Schmid \cite{sch73}, and Steenbrink-Zucker \cite{zbMATH04017061}, which is used multiple times (sometimes implicitly). For instance, these assumptions appear in \cite[Thm. 1]{MR1154159} where an admissible $\mathbb{Z}$VHS is called \emph{good}. 

First, we introduce the following:
\begin{defn}\label{mondef}
Let $S$ be a smooth (irreducible) complex variety and $s_0\in S(\C)$. The \emph{algebraic monodromy} of a $\Z$VHS $\V$ on $S$ is the connected component of the identity of the $\Q$-Zariski closure of the image of the monodromy representation 
\begin{displaymath}
    \pi_1(S(\C),s_0)\to \GL(\V_{s_0})
\end{displaymath}
associated to the underlying local system $\V_\Q$.
\end{defn}
We denote the algebraic monodromy group by $\mathbf{M}=\mathbf{M}(\V)$, or sometimes by $\mathbf{M}_{s_0}$, and recall here that it is a $\mathbb{Q}$-normal subgroup of the derived subgroup of the generic Mumford-Tate group of $\V$, as established in \cite[Thm. 1]{MR1154159}.

\begin{lem}
\label{genRlem}
Let $K \subset \mathbb{C}$ be a field, and let $\V$ be a $\mathbb{Z}$VHS on a $K$-algebraic open subvariety $U \subset \mathbb{A}^{\ell}$. Fix a point $p \in U(K)$. Then there exists a $K$-rational curve $R$ passing through $p$ which is Hodge generic for $\V$.
\end{lem}

\begin{proof}
Let $g : \mathcal{R} \to T$ be the family of all curves in $U$ obtained as intersections with $U$ of one-dimensional affine linear subspaces of $\mathbb{A}^{\ell}$ passing through $p$. By the Lefschetz hyperplane theorem, there is a non-empty $K$-algebraic open subset $T^{\circ} \subset T$ such that, for each fibre $g^{-1}(t)$ with $t \in T^{\circ}$, the map $\pi_1(g^{-1}(t), p) \to \pi_1(U, p)$ is surjective. Let $t\in T^{\circ}(K)$ and set $R=g^{-1}(t)$. By construction $R$ has the same algebraic monodromy of $U$. We claim that this implies that they have the same generic Mumford--Tate group. By a theorem of Chevalley (see \cite[Prop. 3.1]{zbMATH03853242}) every algebraic group is the stabilizer of a line in some faithful linear algebraic representation. Recall also that the algebraic monodromy group $\mathbf{M}_{R}$ of $\restr{\V}{R}$ is a subgroup of the Mumford-Tate group $\mathbf{G}_{R}=\MT(\restr{\V}{R})$ of $\restr{\V}{R}$ (cf. \cite[Thm. 1]{MR1154159} and the theorem of the fixed part). Let $L$ be a rational line (in some tensorial representation) fixed by $\mathbf{G}_R$. Notice that $L$ is invariant under the algebraic monodromy of $\mathbf{M}_R$. By construction of $R$, it follows that $L$ is fixed by the algebraic monodromy of $U$. The theorem of the fixed part implies that $L$ extends (up to passing to a finite covering) to a flat subbundle of the corresponding bundle used to construct $L$, and sections of this subbundle are moreover Hodge everywhere on $U$. In particular, $L$ is fixed by the generic Mumford-Tate group $\MT(\V)$ of $U$ and so $\mathbf{G}_R= \MT(\V)$, as desired.

\end{proof}

The next lemmas are needed to control the representations needed to define Mumford-Tate subgroups associated to a fixed graded-polarized lattice $(V, W_{\bullet}, Q)$.

\begin{lem}
\label{MTgroupfamlem}
Let $(V,W_\bullet,Q)$ be a graded-polarized lattice and let
$\mathcal S$ be a prescribed collection of Hodge numbers. Then there exists a finite collection of $\mathbb Q$-algebraic families $\mathcal{F}_{i} \to \mathcal{Y}_{i}$, with $\mathcal F_i\subset\operatorname{GL}(V)\times\mathcal{Y}_i$, of algebraic subgroups preserving $W_\bullet$ such that every
Mumford--Tate group associated to $(V,W_\bullet,Q,\mathcal S)$ is a
fibre $(\mathcal F_i)_y$ for some $y\in\mathcal Y_i(\mathbb Q)$.
\end{lem}
\begin{proof}
Choose a rational splitting of $W_\bullet$ and write
\[
 \operatorname{GL}(V)_{W_\bullet}=U\rtimes L,
 \qquad L=\prod_m\operatorname{GL}(\operatorname{gr}^W_mV).
\]
The reductive groups $\MT(\operatorname{gr}^{W_{\bullet}} V)$
belong to finitely many $L(\mathbb C)$-conjugacy classes: apply
 \cite[Prop. 4.15]{zbMATH06492665} to the associated graded pure Hodge structure $\operatorname{gr}^{W} V$, whose Hodge
cocharacters have the finitely many possible conjugacy types determined by $\mathcal S$. Choose rational representatives $R_i$ for the conjugacy classes that occur. Then the coset spaces $L/N_L(R_i)$ carry universal families of
conjugates of $R_i$, and moreover, every rational subgroup in such a class gives a rational point of this quotient (regardless of whether there exists a rational element of $L(\mathbb{C})$ which conjugates it to $R_{i}$).

Over these parameter spaces, for each $R = R_{i}$, the $R$-stable Lie subalgebras
$\mathfrak n\subset\operatorname{Lie}(U)$ are parameterized by
closed loci in finitely many relative Grassmannians, defined by
\[
 [\mathfrak n,\mathfrak n]\subset\mathfrak n,
 \qquad [\operatorname{Lie}(R),\mathfrak n]\subset\mathfrak n.
\]
Taking the subgroups $\exp(\mathfrak n)\rtimes R$ and all their $U$-conjugates gives a finite-type family over $\mathbb Q$. Here we note stability under the adjoint action of $R$ is equivalent to stability of the associated algebraic group under $R$-conjugation because $R$ is connected. By rational Levi
decomposition, every mixed Mumford--Tate
group occurs as a rational fibre of the resulting collection of finite-type families.
\end{proof}

\begin{lem}\label{lem:chev} 
 Let \((V,W_\bullet,Q,\mathcal S)\) be as in \Cref{MTgroupfamlem}, and put
\[
T_{\le d}(V):=
\bigoplus_{a,b\geq 0, a+b \leq d}
V^{\otimes a}\otimes(V^\vee)^{\otimes b}.
\]
There exists an integer $d\ge1$ such that every Mumford–Tate group associated to a graded-polarized mixed Hodge structure on $V,W_\bullet,Q$ with Hodge numbers $ \mathcal S$ is the setwise stabilizer in $\mathrm{GL}(V)$ of a rational line in $T_{\le d}(V)$.
\end{lem}

\begin{proof}

By \Cref{MTgroupfamlem}, it suffices to prove the assertion for each of the
finitely many families
\[
g:\mathcal F\to\mathcal Y,
\qquad
\mathcal F\subset\mathrm{GL}(V)\times\mathcal Y,
\]
appearing there. We argue by induction on $\dim\mathcal Y$.
Treating the irreducible components separately, we may assume that
$\mathcal Y$ is integral. Let $\eta$ be its generic point, with
residue field $\kappa(\eta)=\Q(\mathcal Y)$. By Chevalley's theorem (cf. the proof of \Cref{genRlem}), the generic fibre $\mathcal F_\eta$ is the stabilizer of a $\kappa(\eta)$-rational line in a finite-dimensional rational representation of $\mathrm{GL}(V)_{\kappa(\eta)}$.
Such a representation embeds into a finite direct sum of mixed
tensor powers of $V_{\kappa(\eta)}$.

Thus, for some integer $d$, there is a line $L_\eta\subset T_{\le d}(V)_{\kappa(\eta)}$ whose stabilizer is precisely $\mathcal F_\eta$. Over a nonempty open subset $U\subset\mathcal Y$, this line
extends to a line subbundle
\[
\mathcal L\subset
T_{\le d}(V)\otimes_{\mathbf Q}\mathcal O_U.
\]
After shrinking $U$ further, its stabilizer equals $\mathcal F|_U$.
Indeed, these are closed subgroup schemes of
$\mathrm{GL}(V)\times U$ with equal generic fibres, so their
finitely generated defining ideals agree after restriction to a
suitable nonempty open subset. Consequently, for every
$y\in U(\Q)$, the rational line $\mathcal L_y$ has
stabilizer $\mathcal F_y$. The proof follows by notherian induction on $\mathcal{Y}$.
\end{proof}

\begin{lem}
\label{finitelymanyNLequivlem}
Let $\V$ be a $\Z$VHS on $S$ with associated flag variety $\ch{D}$. Let $\ch{D}^0= \mathbf{M}_{s_0}(\C) \cdot \psi (s_0)$ be the monodromy orbit of some base point $s_0$ (where $\psi$ is as in \eqref{psidef}). There exists a family $g : Y \to W$ of algebraic subvarieties of $\ch{D}$, with $W$ possibly disconnected but of finite-type, with the following property:
\begin{itemize}
    \item For each sub-Mumford-Tate domain $\ch{D}'$ of $\ch{D}^0$, there is a $\mathbb{Q}$-point $w \in W$ such that $\ch{D}'$ appears as a component of $g^{-1}(t) \cap \ch{D}^0$.
\end{itemize}
\end{lem}

\begin{proof}
Since each sub-Mumford-Tate domain of $\ch{D}^0$ is obtained as the intersection of $\ch{D}^0$ with a sub-Mumford-Tate domain of $\ch{D}$, it is enough to find a family of subvarieties of $\ch{D}$ such that all sub-Mumford-Tate domains of $\ch{D}$ appear as components of its $\Q$-fibers. Each such sub-Mumford-Tate domain $\ch{D}''\subset \ch{D}$ is a component of a locus where the associated Hodge flags contain an extra $\Q$-line in some faithful linear $\Q$-representation of $\GL (\V_{s_0})$. \Cref{lem:chev} shows that it is enough to consider lines associated to finitely many representations, and hence finitely many algebraic families, and so there are only finitely many families of subvarieties of $\ch{D}$ to be considered.
\end{proof}

\section{Algebraic Hodge generic points and applications}\label{sec6}

In this final section, we prove the results announced in \Cref{sec1} and \Cref{section2}.

\subsection{General setting and $K$-geometric $\Z$VHS} \label{setting}

\subsubsection{Geometric variations of Hodge structures}

\begin{defn}
 Let $K\subset \Qbar \subset \C$ be a field and $S$ a smooth and geometrically connected $K$-variety. A \emph{$K$-structure} on a $\Z$VHS $\V = (\V_\Q, \mathbb{W}_{\bullet}, \mathcal{V}, F^\bullet, \nabla, \mathcal{Q})$ on $S(\C)$ is the data of an algebraic model over the $K$-variety $S$ of the complex vector bundle $\mathcal{V}$, its subbundles $(F^\bullet, \mathbb{W}_{\bullet}\otimes_{\Q_{S(\C)}} \mathcal{O}_{S(\C)})$, and the connection $\nabla$, compatible with the natural inclusions and identifications.
\end{defn}

\begin{rem}
    Given a smooth projective morphism $f: X \to S$ defined over $K$, each $R^jf_*\Q_{S(\C)}$ is naturally a $\Z$VHS with $K$-structure.
\end{rem}

\begin{defn}\label{everythinggeneralizesrem} 
Let $K \subset \Qbar \subset \C$ be a field and $S$ a smooth and geometrically connected $K$-variety. A $K$-\emph{geometric} $\Z$VHS is a $\Z$VHS with a $K$-structure obtained by taking direct sums, direct factors, tensor products, and iterated extensions of polarizable $\mathbb{Z}$VHS with $K$-structure arising from the cohomology of smooth proper families $X\to S$ defined over $K$. 
\end{defn} 

Consider the case where $S$ is a (non-empty) open subvariety of $\mathbb{P}^{1}$ and let $\V$ be a $K$-geometric VHS on $S$.  Associated to the connection of $\V$ there is a differential operator $L$. Moreover, as explained in \Cref{thmGoperator}, $L$ is a $G$-operator, since it is obtained by successive extensions of (pure quotients of) algebraic de Rham vector bundles with connection. A $\Z$VHS with an associated $G$-operator is all we need for our applications of \Cref{thmwithproof}. 

\subsubsection{Quasi-projective families defined over $K$ are sources of $K$-geometric $\Z$VHS}
\label{VHSconstrsec}

Let $f: X \to S$ be any quasi-projective family of algebraic varieties. By a theorem of Verdier \cite[Cor. 5.1]{zbMATH03520733} given \emph{any} family $f : X \to S$ of complex algebraic varieties (that is, any map of algebraic varieties), there is a non-empty Zariski open subset $U \subset S$ such that $f^{-1}(U) \to U$ is a topological fibration in the complex analytic topology; we therefore assume $f$ is a locally trivial topological fibration, as this will not affect any density results we wish to prove. It then follows from Saito's theory of mixed Hodge modules \cite{zbMATH04200363}, and in particular the stability of polarizable mixed Hodge modules under direct images of quasi-projective morphisms, that the cohomology local systems $R^{j} f_{*} \mathbb{Q}$ admit the structure of a graded-polarizable mixed $\mathbb{Z}$VHS satisfying Griffiths transversality and the admissibility condition. If, furthermore, the family is defined over $K$, Saito's construction shows that $R^jf_*\Q$ is, in fact, a $K$-geometric $\Z$VHS. (We remark here also that it is known that the open subset appearing in Verdier's theorem can also be defined over $K$ if $f$ is, as follows from the theory of Whitney stratifications over $K$, see, e.g., \cite[Lem. 3.1.9.]{zbMATH06126017}.)

\subsection{Abundance of algebraic Hodge generic points -- proof of \Cref{mainhodgegenthm}}\label{sec6.2}

In this section, we prove \autoref{mainhodgegenthm} and \Cref{effectivethmintro}, using \autoref{thmwithproof} and the results from \S \ref{lemmas}. For the rest of the section we let $K\subset \Qbar \subset \C$ be a number field and $S$ a smooth and geometrically connected $K$-variety.

\begin{thm}
\label{generalHggenthm2}
Let $\V$ be a $K$-geometric $\Z$VHS on $S$. There exists an integer $d$ such that the subset
\[ \operatorname{HgGen}(S,\V,d) := \{ s \in S(\overline{\mathbb{Q}}) : [\kappa(s) : K] \leq d, \hspace{0.5em} s \notin \HL(S, \V^{\otimes}) \} \]
is analytically dense in $S(\mathbb{C})$.
\end{thm}

\begin{rem}
If $S$ is a curve and $K$ is dense in $\mathbb{C}$, our proof shows that the integer $d$ in \autoref{generalHggenthm2} can be taken to be the $K$-gonality of $S$, i.e. the minimum degree of a nonconstant $K$-morphism to $\PP^1$.
\end{rem}

\autoref{generalHggenthm2} follows from \autoref{thmwithproof}(1), as we now explain. Using the more precise \autoref{thmwithproof}(3), we will also get the following, which is, in fact, essential to control the density of the various intersection loci appearing in \Cref{mainthmhyper}:

\begin{thm}\label{quantitativeHodgegen}
Let $\V$ be a $K$-geometric $\Z$VHS on $S$, where $S$ is an open subvariety of $\PP^1$. Among the points of $\PP^1(K)$ of the form $[t:1]$ with $t$ an integer, the subset of such points which lie in $\operatorname{HgGen}(S,\V,1)$ has natural density $1$ in $\mathbb{N}$.
\end{thm}

\begin{proof}[Proof of \Cref{generalHggenthm2}]
Arguing by Noetherian induction, we may assume that $S$ has an \'etale map $S \to \mathbb{A}^{\ell}$ for some $\ell$ with image $U$. Thanks to \cite[\href{https://stacks.math.columbia.edu/tag/02NW}{Lemma 02NW}]{stacks-project}, we may replace $U$ with a Zariski open and $S$ with its inverse image so that $e : S \to U$ is additionally finite, hence finite \'etale. Consider the $K$-geometric VHS $e_{*} \V$. Since the degree of $e$ is bounded by some integer $k$, if $d$ is divisible by $k$ we can prove the density of $\operatorname{HgGen}(S,\V,d)$ by proving the density of $\operatorname{HgGen}(U, e_{*} \V, d/k)$. Thus we can assume that $S$ is a rational open subset of some $\mathbb{A}^{\ell}$. If $K$ is not dense in $\C$ we may adjoin $\sqrt{-1}$ so that it is, and hence reduce to the case where $d = 1$.

Let $p \in S(K)$, and choose using \autoref{genRlem} a Hodge generic rational curve $R \subset S$ (in the sense of \Cref{hodgegen}) which is defined over $K$ and passes through $p$. We may reduce to the case where $S = R$. It suffices to prove that Hodge generic points accumulate in $S$ near $p$. We let $S$ be the $\PP^1 - \Sigma$ of \autoref{thmwithproof}, such that $p$ is identified with $\infty$, and so $s$ is a uniformizing parameter at $p$.

Let $\varphi$ be the period map associated to $\V$, as in \S\ref{vhsandpermapsec}. We let $\mathbf{P}(s)$ be the Betti-de Rham period matrix, as constructed in \S\ref{vhsandpermapsec}. Let $L$ be the $G$-operator associated to the connection on $\V$, cf. the discussion after \Cref{everythinggeneralizesrem}. This period matrix may also be regarded as a solution of $L$ in the sense of \S\ref{Gopsec}. In particular, we may apply \autoref{thmwithproof} to conclude that, for each integer $\delta$, there are infinitely many values of $s = 1/(t + jD)$, with $t+jD$ an integer, such that $\mathbf{P}(s)$ does not satisfy any non-trivial $K$-algebraic relations of degree at most $\delta$. We now show that, taking $\delta$ large enough, this implies that all such values of $s$ are Hodge generic. It suffices to show that there is a uniform $\delta$, such that whenever $s$ is not Hodge generic, then $\mathbf{P}(s)$ satisfies a non-trivial relation of degree $\delta$.

Recall that a $K$-relation $H=H_R$ on $\mathbf{P}(s)$ is said to be non-trivial if $R$ does not vanish on any $K$-irreducible component of $B_{\mathbf{P},s}$ passing through $\mathbf{P}(s)$; cf. \Cref{def:nontrivial}. Write $M_{\mathbf{P}}$ for the $\mathbb{C}$-Zariski closure of the graph of $\mathbf{P}$; of course $M_{\mathbf{P}} \subset B_{\mathbf{P}}$. To produce a non-trivial relation on $\mathbf{P}(s)$, it will suffice to produce a relation $H$ vanishing at $\mathbf{P}(s)$ such that $\dim(M_{\mathbf{P},s} \cap H) < \dim M_{\mathbf{P}} - 1$; such a relation is automatically non-trivial.
Recall that $\mathbf{P}$ was constructed with respect to a fixed filtration-compatible basis $b_{1}, \hdots, b_{n}$ of $\V(B)$, where $B \subset S(\C)$ was some simply connected open subset. Let $(V, W_{\bullet}, Q)$ be as in \S\ref{vhsandpermapsec}, and write $\ch{L}$ for the full flag variety of all flags on $V$ with the same Hodge numbers as those of $\V$; note that $\ch{D}$ is a closed subvariety of $\ch{L}$, and we have a natural map $e : B_{\mathbf{P}} \to \ch{L}$ which sends the graph of $P$ into $\ch{D}$. Note that the image of $P(s)$ in $\ch{D}$ actually lands in $\ch{D}^{0}$ for each $s$; this is explained in \cite[\S4.6]{2024arXiv240616628B}. From \autoref{finitelymanyNLequivlem} we see that there exists a family $g : Y \to W$ of algebraic subvarieties of $\ch{D}$ (and hence $\ch{L}$), with $W$ possibly disconnected but of finite-type, such that $\mathbf{P}(s)$ lies in a $\mathbb{Q}$-fibre of $g$ whenever $s$ is not Hodge generic. Moreover, we may assume that each fibre of $g$ intersects the image of $M_{\mathbf{P},s}$ under $e$ properly: that is we have
\[ \dim (M_{\mathbf{P},s} \cap (g \circ e)^{-1}(w)) < \dim M_{\mathbf{P},s} \]
for each $w \in W$ and $s \in S$.

Because $W$ is finite-type, the ideals defining the fibres of $e \circ g$ are given by generators whose degrees are uniformly bounded by some integer $\delta$. Applying \autoref{thmwithproof}(1) with this $\delta$ we may conclude.
\end{proof}
(The idea of bounding the degree of the exceptional relations in terms of the degree of exceptional tensors appears also, in a special case, at the very end of \cite[(page 58)]{zbMATH00764346}.  In the above proof, \autoref{finitelymanyNLequivlem} plays indeed a similar role.)

\begin{proof}[Proof of \Cref{quantitativeHodgegen}]
Let $x$ be the affine coordinate corresponding to the affine space $\mathbb{A}^{1} \hookrightarrow \mathbb{P}^1$ embedded via $x \mapsto [x : 1]$, and set $s = 1/x$. Carry out the preceding argument on a slit punctured disk centered at $s=0$; note that $s = 0$ (i.e. $x = \infty$) may lie in $\Sigma$. 

For each $\nu\geq1$, \autoref{thmwithproof}(3) supplies an integer $D$ such that, for every sufficiently large integer $t$, at least $\nu$ of the $n^2+\nu+1$ integers
\[
 t,t+D,\ldots,t+(n^2+\nu)D
\]
correspond to Hodge generic points. Summing over $1\leq t\leq X$, each integer is counted at most $n^2+\nu+1$ times. Thus the lower density of Hodge-generic integers is at least $\nu/(n^2+\nu+1)$. Letting $\nu\to\infty$ proves the claim.
\end{proof}

\subsubsection{Effective strengthening}
\label{effectsec}
So far we have used \autoref{thmwithproof} only as a density statement, but in fact \autoref{thmwithproof}(3) implies a \emph{finer effective} result concerning which parameter values can be Hodge generic. The following implies \Cref{effectivethmintro}:
\begin{cor}
\label{effcor}
Let $\V$ be a $K$-geometric $\Z$VHS on $S$, with $S$ an open subvariety of $\PP^1$. Let $\alpha \in [0,1) \in \mathbb{Q}$. There exist three integers $\kappa, D, N$ depending on $\alpha$, effectively computable from the equations defining the filtered vector bundle with connection on $S$, such that when $t > \kappa$ at least an $\alpha$-fraction of the points in the subset
\[ \{ [t : 1], [t + D : 1], \hdots, [t + DN : 1] \} \subset S(K)\]
are Hodge generic.
\end{cor}

For a general discussion on how to compute the Gauss-Manin connection of a given family of varieties, we refer to \cite[\S2]{urbanik2021sets}, \cite[Prop. 9.4]{2024arXiv240616628B}.
\begin{proof}
Follow the proof of \autoref{quantitativeHodgegen} until the invocation of \autoref{thmwithproof} at the very end, and with $s$ a parameter centered at $\infty$. In this case, one invokes \autoref{thmwithproof}(3) for $\nu$ large enough so that $\alpha<\frac{\nu}{n^2+\nu+1}$, and take $D=D(\nu)$, $\kappa=\kappa(\nu)$ and
$N=n^2+\nu$ from \Cref{thmwithproof}(3). To see effectivity, one follows the proof of \autoref{thmwithproof} to control the constants $\kappa$ and $D$. The integer $D$ comes from \autoref{lemmashift}, and requires testing finitely many disjointness conditions $jD \neq \sigma_1 - \sigma_2$. The constant $\kappa$ comes from \S\ref{heightboundsec}, where it is chosen using \autoref{thm:bombieri} so that the entries of $\mathbf{H}$ in \S\ref{heightboundsec} have no points $s = 1/t$ with $t > \kappa$ admitting non-trivial $K$-relations of degree $\rho$, where $\rho$ is as in \S\ref{heightboundsec}. Since a system of differential equations for $\mathbf{H}$ can be derived from the original Gauss-Manin system, and $\rho$ can likewise be computed from an explicit presentation of the associated differential modules, everything reduces to determining the constants $c_{1}, c_{2}$ in \autoref{thm:bombieri}. That these constants are effective is well-known and explained in \cite[\S VII]{zbMATH00041964} with explicit formulas (cf. \cite[\S11, \S12]{zbMATH03721021}). 
\end{proof}

\subsection{Motivic applications}
We now prove the motivic applications described in \Cref{sec1} and \Cref{section2} (for smooth projective families of varieties). We note here that such results have natural effective strengthenings as in \Cref{effectsec}, since, in fact, they are deduced from the same theorem.

\subsubsection{The period conjecture up to a certain degree}\label{proofperiod}

\begin{proof}[Proof of \Cref{galperiodconj}]
    This follows directly from \autoref{thmwithproof}(1) and the fact that geometric VHS give rise to $G$-operators. Indeed, as in the proof of \Cref{generalHggenthm2}, the statement is reduced to the case where the base $S$ is a Zariski open of $\mathbb{P}^1$ and then one applies \autoref{thmwithproof}(1) to the $G$-operator associated to the Gauss-Manin connection. Note that the condition in \autoref{degdeltaperconj} that a relation does not vanish on any irreducible component of $B_{\mathbf{P},s}$ agrees with the definition of non-triviality appearing in \autoref{def:nontrivial}.
\end{proof}
Under assumption \eqref{bigmon}, we can also revisit our results on the period conjecture and state them in more classical terms. Note that under \eqref{bigmon}, the locus $S_{P,j,\delta}$ appearing in \Cref{galperiodconj} can be reinterpreted as the locus where all period relations up to degree $\delta$ are accounted for by the generic absolute Mumford-Tate group; in particular, \autoref{perconj} is just a reformulation of the Grothendieck period conjecture for Hodge generic points. Moreover if, for a fixed Hodge generic point $s \in S(\Qbar)$, \autoref{degdeltaperconj} holds for all integers $\delta$, then in fact \autoref{perconj} holds: given a $\Qbar$-relation $R$ of some degree $d$ vanishing on $\mathbf{P}(s)$, one can take a product of $[\kappa(R) : \kappa(s)]$ conjugates of $R$ to obtain a relation $\mathcal{R}$ of degree $\delta = d \cdot [\kappa(R) : \kappa(s)]$ defined over $\kappa(s)$. Since $B_{\mathbf{P
},s}$ is irreducible in this setting (it is a torsor for the generic Mumford-Tate group under assumption \eqref{bigmon}), the relation $\mathcal{R}$ is also non-trivial.

We obtain the following corollary of \autoref{galperiodconj}.
 
\begin{cor}\label{mainthmonperiod}
Let $f : X \to S$ be a smooth projective family defined over a number field $K \subset \Qbar \subset \mathbb{C}$ and $j$ a natural number. If $[R^jf_*\Q]_{\operatorname{prim}}$ satisfies \eqref{bigmon} then, for any $\delta \geq 1$, there is a set of points in $S(\Qbar)$, dense in $S(\mathbb{C})$, where the Grothendieck period conjecture holds up to degree $\delta$.
\end{cor}

\begin{proof}[Proof of \Cref{mainthmonperiod}]
    This will follow from \Cref{galperiodconj} as long as we can show that the variety $B_{P}$ in the statement of \Cref{galperiodconj} is a torsor for the generic Mumford-Tate group. It suffices to prove the statement after replacing $S$ with a finite \'etale covering, so we may assume that the image of monodromy lands inside the generic Mumford-Tate group. We recall the construction of $\mathbf{P}$. One fixes a frame $v_{1}, \hdots, v_{n}$ for the algebraic de Rham cohomology bundle $[R^{j} f_{*} \Omega^{\bullet}_{X/S}]_{\operatorname{prim}}$ and a basis $b_{1}, \hdots, b_{n}$ for the fibre of $\V = [R^{j} f_{*} \mathbb{Q}]_{\operatorname{prim}}$ at some point $s_{0} \in S(\mathbb{C})$, and lets $P$ be the graph of the change-of-basis isomorphism. Let $\mathbf{G}_{s_{0}}$ be the realization of the generic Mumford-Tate group in $\GL(\V_{s_{0}})$. Then if $t$ is a tensor invariant of $\mathbf{G}_{s_{0}}$, $t$ extends to a flat section $\widetilde{t}$ of a local system in the Tannakian category generated by $\V$, and has a corresponding de Rham realization $\widetilde{t}_{\operatorname{dR}}$. Moreover, as a consequence of our assumption (\ref{bigmon}), the section $\widetilde{t}_{\operatorname{dR}}$ is defined over (a finite extension of) $K$. Then the relation on $(s, A) \in S \times \GL_{m}$ given by $A \cdot \widetilde{t}_{\operatorname{dR}}(s) = t$ is a $K$-algebraic relation on the graph of $\mathbf{P}$. Ranging over all such tensor invariants $t$, one sees that $B_{\mathbf{P}}$ lies inside a torsor $B_{\operatorname{MT}}$ for the Mumford-Tate group $\mathbf{G}_{s_{0}}$.

    On the other hand, by analytic continuation, $B_{\mathbf{P}}$ contains a torsor $B_{M}$ for the (complex points of the) algebraic monodromy group $\mathbf{M}$. Since $\mathbf{M} = \mathbf{G}_{s_{0}} / \mathbb{G}_{m}$, it suffices to show that the extra relation of the form $\det \mathbf{P}(s) = F(s)$, where $F(s)$ is some algebraic function, is not defined over $K$ unless $j = 0$. Specializing to some $\Qbar$-point $s \in S(\Qbar)$, one observes via the standard computation (cf. \cite[\S1]{zbMATH03853242}) that $F(s)$ is a $\Qbar$-multiple of a power of $\pi$, and therefore not $\Qbar$-algebraic.
\end{proof}

\begin{proof}[Proof of \Cref{hypercor}:]
Apply \autoref{mainthmonperiod} to the $\mathbb{Z}$VHS in degree one associated to the specified family of hyperelliptic curves, where the base $S$ is $\mathbb{P}^1 - \Delta$ where $\Delta$ is the discriminant of $f(x)(x-t)$. Note that the effective strengthening of \autoref{mainthmonperiod} follows from a parallel argument to \autoref{effcor}, which we omit. That the assumption of \eqref{bigmon} holds follows from an unpublished theorem of J-K. Yu. See also \cite[Thm. 1.2]{zbMATH05224877}.
\end{proof}

\subsubsection{The Mumford-Tate Conjecture: proof of \Cref{mainthmonMT}}\label{proofofmt}
Fix a prime $\ell$. We first treat a family over an open subset
$S\subset\mathbb P^1_K$. Set $s=1/t$ and work with the primitive
variation $\mathbb V=[R^jf_*\mathbb Q]_{\mathrm{prim}}$. We will
show that the points satisfying the $\ell$-adic Mumford--Tate
conjecture for $\mathbb V_t$ accumulate near $s=0$.

As explained in \Cref{relworksec}, Serre's version of the Hilbert irreducibility theorem shows all points of $\PP^1(K)$ outside of a thin set have $\ell$-adic monodromy ``as big as possible''; let us recall what this means.

First, we recall that a thin subset of $\PP^1(K)$ is a subset given as the union of the images of finitely many maps $C(K) \to \PP^1(K)$, where $C \to \PP^1$ is a ramified covering of degree $> 1$. Regarding $\ell$-adic monodromy, recall that we have a split exact sequence
\begin{equation}
\label{galoisexactseq}
\begin{tikzcd}
1 \arrow[r] & \pi^{\textrm{\'et}}_1(S_{\overline{K}}, \overline{t}) \arrow[r] & \pi^{\textrm{\'et}}_1(S, \overline{t}) \arrow[r] & \textrm{Gal}(\overline{K}/K) \arrow[r] \arrow[bend right=30, swap, "\sigma_{t}"]{l} & 1
\end{tikzcd}
\end{equation}
with the section labeled ``$\sigma_{t}$'' corresponding to a choice of $K$-point $t \in S(K)$. The family $f : X \to S$ over $K$ induces a lisse $\ell$-adic sheaf $[R^j f_{\'et, *} \Q_\ell]_{\operatorname{prim}}$, which can be identified with a representation $\rho : \pi^{\textrm{\'et}}_{1}(S, \overline{t}) \to \GL(H^{j}_{\'et}(X_{\overline{t}}, \mathbb{Q}_{\ell})_{\operatorname{prim}})$. By \cite[\S10.6]{zbMATH00042767} (cf. \cite[Fact 3.3.1.1]{zbMATH06657570}) there exists a thin subset of $S(K)$ outside of which $\rho$ and $\rho \circ \sigma_{t}$ have the same image.

Observe that the intersection with $\mathbb{N} = \{ [t : 1] \, | \, t \in \mathbb{Z} \} \subset \PP^1(K)$ and a thin subset $\mathcal{T}$ of $\PP^1(K)$ has density zero: indeed the number of integral points of height at most $H$ of $\mathcal{T}$ is $O(H^{1/2})$ (cf. \cite[p. 27]{zbMATH00053635}). By \Cref{quantitativeHodgegen}, the integral values of $t$ that are Hodge generic have natural density one in $\mathbb{N}$, so we obtain an infinite subset of $S(K)$, accumulating near $s = 0$, where both the Mumford-Tate group and the Galois action are as large as possible.

Now the generic Mumford-Tate group at $t$ has the algebraic monodromy group $\mathbf{M}_{t} \subset \GL(H^j(X_{t,\overline{\mathbb{C}}}, \mathbb{Q}))$ as a subgroup, and the Zariski closure of 
\[ \textrm{im} \left[ \pi^{\'et}_{1}(S, \overline{t}) \to \GL(H^j_{\'et}(X_{\overline{t}}, \mathbb{Q}_{\ell})_{\operatorname{prim}}) \right] \]
has as a subgroup the Zariski closure of the image of $\pi^{\'et}_{1}(S_{\overline{K}}, \overline{t})$. Under the Artin comparison theorem $H^{j}(X_{t,\mathbb{C}}, \mathbb{Q})_{\operatorname{prim}} \otimes \mathbb{Q}_{\ell} \xrightarrow{\sim} H^{j}(X_{\overline{t}}, \mathbb{Q}_{\ell})_{\operatorname{prim}}$ these two geometric monodromy groups agree by \cite[Lem. 3.3]{zbMATH07532081} (cf. \cite{zbMATH07144489}). 

The case where $\ell$ is not prime follows similarly by carrying out the argument for each prime factor of $\ell$ and observing that the intersection of finitely many density one sets still has density one. The positive-dimensional general base case follows analogously by reducing to this case by d\'evissage.

\begin{rem}
If one assumes only the validity of the first equality of \eqref{bigmon}, i.e., that the monodromy $\mathbf{M}$ is equal to the generic Mumford-Tate group up to the diagonal $\mathbb{G}_m$ factor (which is needed whenever the weight is non-zero), the above argument shows the abundance of points $t \in S(K)$ where $\MT(X_t)_{\Q_\ell}$ is naturally contained in the Zariski closure of the $\ell$-adic monodromy at $s$. Indeed, one obtains that, up to the action of $\mathbb{G}_{m, \Q_\ell}$ corresponding to the cyclotomic character, $\MT(X_t)_{\Q_\ell}$ is contained in the Zariski closure of the image of $\pi_1^{\textrm{\'et}}(S_{\overline{K}})$. This is ``half of'' the Mumford-Tate conjecture for $X_t$; the other half would follow if one knew that the global Hodge tensors which define the generic Mumford-Tate group satisfied the $\ell$-adic absolute Hodge property.
\end{rem}

\subsubsection{Families of smooth hypersurfaces}

In this final section, we prove the results announced in the introductory section \S\ref{mmr}, and in particular discuss applications to the Hodge and Tate conjectures. Let $S=U_{n,d}$ be the parameter space of smooth degree $d$ hypersurfaces in $\PP^{n+1}$. This is naturally defined over $\Q$, for more details on its construction see e.g. \cite{BKU2}. We let $\V$ be the polarized $\mathbb{Z}$VHS associated to the primitive cohomology in degree $n$. We would like to show density of the locus
\[
S_{\operatorname{Motivic}, j,\ell, \delta} \;=\;  S_{P,j,\delta}  \;\cap\;S_{\mathrm{H},j} \;\cap\; S_{\mathrm{T},\ell,j} \;\cap\; S_{\mathrm{MT},\ell,j}.
\]
The idea is that each of the loci appearing in the above intersection has density one (after restriction to any rational curve, and counting integer-valued points), and therefore their intersection is again dense.

\begin{proof}[Proof of \Cref{mainthmhyper}]
First, recall that the primitive cohomology in degree $n$ of the universal family of smooth hypersurfaces on $U_{n,d}(\C)$ has large monodromy and so satisfies the assumption \eqref{bigmon}; see indeed \cite[Sec. 9.1]{2021arXiv210708838B} and references therein. Arguing as in the previous sections, we can reduce ourselves to the case of a Hodge generic family of smooth hypersurfaces over $S=\PP^1-\Sigma$, defined over a number field $K\subset \Qbar \subset \C$ (as usual, we fix a parameter $s = 1/t$ on $\mathbb{P}^1$). Consider the loci
\begin{displaymath}
    \N_{\mathrm{MT}}:=   \{ t \in \mathbb{N} : [t : 1] \in S_{\mathrm{MT},\ell,n}\}
 \ \ \text{  and  } \ \
     \N_{P, \delta}:=   \{ t \in \mathbb{N} : [t : 1] \in S_{\mathrm{P},n,\delta}\}.
\end{displaymath}
Since condition \eqref{bigmon} is satisfied, we can apply the density one version of \Cref{mainthmonperiod} (which is argued analogously to the density one result \Cref{quantitativeHodgegen}) to see that  $\N_{P, \delta}$ has natural density one. Moreover, as explained in the proof in \Cref{proofofmt}, $\N_{\mathrm{MT}}$ also has density one. It follows that the subset $\N_{\mathrm{MT}} \cap \N_{P, \delta}$ has density one in $\mathbb{N}$, since it is the intersection of two density one subsets.

From the construction in \Cref{proofofmt}, the hypersurface $X_t$ corresponding to a $t \in  \N_{\mathrm{MT}}\cap \N_{P, \delta}$ is Hodge generic. We are left to show that the (tannakian) Hodge and Tate conjectures hold for $X_t$. Since we already know that $X_t$ satisfies the Mumford-Tate conjecture, it is enough to show that the Hodge conjecture holds. The fact that a Hodge generic hypersurface satisfies the (tannakian, up to a certain degree in the even dimensional case) Hodge conjecture is proven in \Cref{lem:HC-generic} below. We conclude that the intersection \eqref{Smotdef} admits points arbitrarily close to $s = 0$ in $\PP^1$. Since the curve $S$ and coordinate $s$ were arbitrary, we conclude by d\'evissage the desired analytic density in $U_{n,d}(\C)$.

\end{proof}
\begin{proof}[Proof of \Cref{thmLefschetzpen}]
    Same proof as above, using that Lefschetz pencils have large monodromy (cf. \cite[Exposé XVII]{zbMATH03407568}).
\end{proof}

We use the notation introduced in \Cref{mmr}. The next lemma is certainly well-known, but we give it for completeness. It is an application of  Weyl's Fundamental Theorem of invariant theory for the symplectic and orthogonal groups \cite[Ch. 17,18]{FultonHarris}, \cite[Thm. 1B]{zbMATH04103278}.
\begin{lem}\label{lem:HC-generic}
Let $X \subset \PP^{n+1}$ be a smooth hypersurface of dimension
$n \geq 1$ and degree $d \geq 3$ over $\C$, and set
\[
    V := H^n(X,\Q)_{\operatorname{prim}},
    \qquad N := \dim_{\Q} V.
\]
Suppose that $X$ is Hodge generic in $U_{n,d}$. Then:
\begin{itemize}
    \item if $n$ is even, $X$ satisfies $\mathrm{H}(X,n,N-1)$;
    \item if $n$ is odd, $X$ satisfies $\mathrm{H}(X,n)$.
\end{itemize}

\end{lem}

\begin{rem}\label{rmk6.11}
In the above statement, if $d \leq 2$ or $(n,d)=(2,3)$, then
$\mathrm{H}(X,n)$ holds without any restriction on the tensor degree. Indeed, if $d \leq 2$, $X$ is a projective space or a smooth quadric, and its cohomology is generated by algebraic classes. The same holds for a cubic surface, by the Lefschetz $(1,1)$ theorem. In these cases, products of algebraic classes and Poincar\'e duality give the full tensorial Hodge conjecture.
\end{rem}

\begin{proof}
By \Cref{rmk6.11}, we may assume $(n,d)\ne(2,3)$. Let $h=c_1(\mathcal O_X(1))$. By the Lefschetz hyperplane theorem
and Poincar\'e duality, the cohomology of $X$ outside its primitive
middle part is generated by powers of $h$. Since
$\int_X h^n=d$, the algebraic correspondence
\[
    \Pi_{\operatorname{prim}}
    :=
    \Delta_X-\frac{1}{d}\sum_{k=0}^{n}h^{n-k}\times h^k
    \in \operatorname{CH}^n(X\times X)_{\Q}
\]
acts on cohomology as the projector onto $V$.
Its cycle class, with the usual Hodge-Tate twist, belongs to
\[
    V\otimes V(n)\simeq V\otimes V^\vee
\]
and corresponds to the identity automorphism of $V$. Using this projector, one reduces to considering just those Hodge classes in tensorial representations generated from $V$.

Consider a Hodge class in
\[
    T^{a,b,c}(V)
    :=
    V^{\otimes a}\otimes(V^\vee)^{\otimes b}\otimes\Q(c).
\]
This space has weight $n(a-b)-2c$, so it contains no nonzero
Hodge classes unless  $2c=n(a-b)$.
Using the polarization and writing $s=a+b$, we obtain, whenever
this equality holds,
\[
    T^{a,b,c}(V)
    \simeq V^{\otimes s}(nb+c)
    =V^{\otimes s}(ns/2).
\]
Since the polarization gives an algebraic duality $V^\vee\simeq V(n)$, we may reduce to producing algebraic cycles in $V^{\otimes s}(ns/2)$.

The generic monodromy theorem for hypersurfaces of Beauville \cite{Beau}, together with
Hodge genericity, identifies the Hodge group $\operatorname{Hg}(V)$ of $V$ with
$\operatorname{Sp}(V,Q)$ if $n$ is odd, and with
$\operatorname{SO}(V,Q)$ if $n$ is even, where
\[
    Q:V\otimes V\to \Q(-n)
\]
is the intersection form. It suffices to consider those elements of $V^{\otimes s}(ns/2)$ stabilized by $\operatorname{Hg}(V)$ and show they come from algebraic cycles.

Suppose first that $n$ is odd. Then the weight condition above forces $s$ to be even. By the fundamental theorem of invariant theory for the symplectic group, the Hodge classes in $V^{\otimes s}(ns/2)$ are generated by products of the class induced by $\Pi_{\operatorname{prim}}$ and permutations of the factors. This proves $\mathrm{H}(X,n)$ for $n$ odd.

If $n$ is even, the situation is the same when $s < N$, giving $\mathrm{H}(X,n,m)$ for $m = N-1$. (The situation changes when $s = N$, as explained in the following remark.)
\end{proof}

\begin{rem}
In even dimension, the argument does not in general give the
full tensorial Hodge conjecture. Indeed, taking the determinant of $V(nN/2)$ induces a one-dimensional Hodge subspace
\[
    \bigwedge^N V(nN/2)
    \subset V^{\otimes N}(nN/2).
\]
A rational generator of this subspace is invariant under the special orthogonal group, but changes sign under an orthogonal transformation of determinant $-1$. This generator therefore cannot be obtained using tensorial constructions derived from $\Pi_{\operatorname{prim}}$, as $\Pi_{\operatorname{prim}}$ is invariant under the full orthogonal group.
\end{rem}

\newcounter{savedsection}
\setcounter{savedsection}{\value{section}}

\appendix
\renewcommand{\thesection}{A}
\section{Independence of logarithms \`a la Siegel}\label{app}
\renewcommand{\thesection}{\arabic{section}}

In his seminal manuscript \cite{zbMATH02563202} (cf. \cite[\S1.4 VIII]{zbMATH06385898}), Siegel wrote:

\begin{quote}
\emph{...there exist infinitely many positive rational numbers $r_1, \hdots , r_n$ so that no algebraic relation among $\log r_1, \hdots, \log r_n$ with rational integer coefficients and bounded degree holds; in particular this holds for $n$ logarithms linearly independent over a given field.}
\end{quote}
\noindent Siegel deduced this from a $G$-function principle that would later be proven by Bombieri \cite{zbMATH03721021}. With hindsight, Siegel's observation is similar to the arguments we employ to prove \autoref{thmwithproof}. To illustrate the relationship, we record here a result, which is essentially a direct application of \autoref{thm:bombieri}, which could have been envisioned by Siegel.

\begin{thm*}
\label{logthm}
For any integers $\delta, k > 0$ there exists a constant $a > 0$ depending on $k$ such that for all $t \in \mathbb{N}$ with $\log t \geq  a \delta^k(1+\log \delta)$ the set of numbers
\[ \left\{ \log\left( \frac{t+1}{t} \right), \hdots, \log\left( \frac{t+k}{t} \right) \right\} \]
satisfies no non-zero polynomial relations over $\mathbb{Q}$ of degree at most $\delta$.
\end{thm*}

(Using \cite[Cor. 4.4 and Lem. 6.10 and 6.11]{orr2026explicitheightboundsgfunctions} one could even compute explicitly $a$ in terms of $k$, but we do not pursue this here.)
\begin{proof}
Set $G_j(z)=\log(1+jz)$ for $1\leq j\leq k$. These are
$G$-functions, algebraically independent over $\mathbb{Q}(z)$.

Write $h=\log t=h(1/t)$. At $\xi=1/t$, the two inequalities in
\autoref{thm:bombieri} become
\[
 h\geq c_1\delta^k(1+\log\delta),
 \qquad h^k\geq c_2^{k+1}\delta^k\log h.
\]
Choose $a\geq\max\{e,c_1,\log(k+1)\}$ such that
$a^k\geq c_2^{k+1}(\log a+k+1)$. With
$L=1+\log\delta$ and $h_0=a\delta^kL$, we have
\[
 \log h_0\leq(\log a+k+1)L,
 \qquad h_0^k\geq a^k\delta^kL.
\]
Thus both inequalities hold for $h=h_0$, and also for $h\geq h_0$,
since $h^k/\log h$ is increasing for $h\geq e$. Moreover
$t\geq k+1$, so $1/t$ lies inside the complex convergence disks
of all the $G_j$.
\autoref{thm:bombieri} proves the claim. 
\end{proof}

One could imagine trying to prove the above theorem in a way similar to the proof of \autoref{thmwithproof}(3) by starting with the differential operator $L_{\operatorname{log},1}$ with solution basis $\{ 1, \log(s) \}$. In this case one would be able to take $D = 1$ in \autoref{lemmashift} and $B_{Q} = \{ 1 \}$ is a point. Note however that the above theorem is stronger than the analogue of \autoref{thmwithproof}(3) in this setting since it guarantees that there are no relations among the different numbers $\log((t + j)/t)$ instead of controlling the transcendence of the logarithms individually.

\setcounter{section}{\value{savedsection}}
\renewcommand{\thesection}{\arabic{section}}

\bibliography{hodge_theory}
\bibliographystyle{abbrv}

\Addresses

\end{document}